\documentclass[10pt]{amsart}
\usepackage{geometry}        
\usepackage{graphicx}
\usepackage{amsmath,amssymb}
\usepackage[foot]{amsaddr}
\usepackage{float}
\usepackage{bbm}
\usepackage{epstopdf}
\usepackage{amsfonts}
\usepackage[utf8]{inputenc}
\usepackage{xcolor}
\usepackage{amsthm}
\usepackage{mathtools}
\usepackage{amsthm}
\usepackage{mathrsfs}
\usepackage{comment}
    
\usepackage{stmaryrd}
\usepackage{xcolor}   
\usepackage{amsmath}
\usepackage{amsthm}
\usepackage{ amssymb }
\newtheorem{theorem}{Theorem}[section]
\newtheorem{definition}{Definition}
\newtheorem{lemma}[theorem]{Lemma}
\newtheorem{proposition}[theorem]{Proposition}

\theoremstyle{remark}
\newtheorem{remark}{Remark}

\usepackage{xcolor}
\usepackage{graphicx}
\usepackage{amssymb}
\usepackage{epstopdf}
\usepackage{enumerate}
\usepackage{bbm}
\usepackage{braket}
\newcommand\abs[1]{\left|#1\right|}
\newcommand{\norm}[1]{ \left\lVert {#1} \right\rVert}
\newcommand{\cadlag}{c\`adl\`ag }

\newcommand{\cP}{\mathcal{P}}
\newcommand{\cR}{\mathcal{R}}

\newcommand{\cC}{\mathcal{C}}
\newcommand{\R}{\mathbb{R}} 
\newcommand{\E}{\mathbb{E}} 
\newcommand{\V}{\mathcal{V}}
\newcommand{\F}{\mathbb{F}}

\newcommand{\al}{\alpha} 

\newcommand{\corr}{\twoheadrightarrow}

\newcommand{\cE}{\mathcal{E}} 
\newcommand{\cQ}{\mathcal{Q}} 
\newcommand{\cD}{\mathcal{D}} 
\newcommand{\cM}{\mathcal{M}}
\newcommand{\cA}{\mathcal{A}}
\newcommand{\tr}{\text{Tr}}

\usepackage{cleveref}
\usepackage{graphicx}
\theoremstyle{definition}
\newtheorem{assump}{Assumption}
\newenvironment{myassump}[2][]
  {\renewcommand\theassump{#2}\begin{assump}[#1]}
  {\end{assump}}
  
 \newtheorem{prop}{Property}
\newenvironment{myprop}[2][]
  {\renewcommand\theprop{#2}\begin{prop}[#1]}
 {\end{prop}}
\usepackage{soul}

\usepackage{subfig}
\usepackage{caption}

\usepackage[backend=biber]{biblatex}
\bibliography{biblio}

\usepackage[normalem]{ulem}
\newcommand{\stkout}[1]{\ifmmode\text{\sout{\ensuremath{#1}}}\else\sout{#1}\fi}

\begin{document}

\title[Mean field games with controlled jump-diffusion dynamics]{Mean field games with controlled jump-diffusion dynamics: Existence results and an illiquid interbank market model} 

\author{Chiara Benazzoli}
\address{Department of Mathematics, University of Trento, Italy}
\email{c.benazzoli@unitn.it}

\author{Luciano Campi}
\address{Department of Statistics, London School of Economics, UK}
\email{l.campi@lse.ac.uk}

\author{Luca Di Persio}
\address{Department of Computer Science, University of Verona, Italy}
\email{luca.dipersio@univr.it}

\date{\today}

\maketitle

\begin{abstract}
We study a family of mean field games with a state variable evolving as a multivariate jump diffusion process. The jump component is driven by a Poisson process with a time-dependent intensity function. All coefficients, i.e. drift, volatility and jump size, are controlled. Under fairly general conditions, we establish existence of a solution in a relaxed version of the mean field game and give conditions under which the optimal strategies are in fact Markovian, hence extending to a jump-diffusion setting previous results established in \cite{La}. The proofs rely upon the notions of relaxed controls and martingale problems. Finally, to complement the abstract existence results, we study a simple illiquid inter-bank market model, where the banks can change their reserves only at the jump times of some exogenous Poisson processes with a common constant intensity, and provide some numerical results. \medskip
\end{abstract}

\section{Introduction}
Mean field games (MFGs, henceforth) were introduced by Lasry and Lions in \cite{lasry2006jeux1,lasry2006jeux2,lasry2007mean} and, independently, by Huang and co-authors in \cite{huang2006large}, as optimization problems approximating large population symmetric stochastic differential games, where the interaction between the players is of mean field type. A solution of the limit MFG allows to construct approximate Nash equilibria for the corresponding
$n$-player games if $n$ is large enough; see, e.g., \cite{carmona2013mean}, \cite{carmona2013probabilistic}, \cite{carmona2015probabilistic}, \cite{huang2006large}, \cite{kolokoltsov2011mean} as well as the recent book \cite{carmona2016lectures}. This approximation result is also practically relevant since a direct computation of Nash equilibria in the $n$-player game with $n$ very large is usually not feasible even numerically, due to the curse of dimensionality. Moreover, MFGs represent a very flexible framework for applications in various areas including but not limited to finance, economics and crowd dynamics (see \cite{gueant2011mean,carmona2016lectures} for a good sample of applications), which partly explain the increasing literature on the subject.\smallskip

In this paper, we study a family of MFGs with controlled jumps, that can be shortly described as follows.
Let $T>0$ be a finite time horizon. We consider the controlled state variable $X=X^\gamma$ taking values in $\mathbb R^d$ and following the dynamics
\begin{equation}\label{Xrelax} dX_t =  b(t,X_t,\mu_t, \gamma_t ) dt + \sigma(t,X_t, \mu_t ,\gamma_t ) dW_t + \beta(t,X_{t-},\mu_{t-},\gamma_{t}) d\widetilde N_t,\quad t\in[0,T],\end{equation}\noindent
where $\mu$ is a probability measure on the Skorokhod space $\cD([0,T];\mathbb R^d)$ of right-continuous with left limit functions, $\gamma_t$ represents a control process with values in a fixed action space $A$, $W$ is a standard multivariate Brownian motion and $\widetilde N$ is a compensated Poisson process with some time-dependent intensity $\lambda(t)$. Moreover, we assume that $W$ and $\widetilde N$ are independent. The precise meaning of $\mu_t$ and $\mu_{t-}$ appearing in the dynamics above will be given in the next section. Intuitively, $(\mu_t)_{t\in [0,T]}$ represents a flow of probability measures on the state space $\R^d$, while $\mu_{t-}$ the left limit at time $t$ with respect to the weak convergence.  

Given some running cost $f$ and some final cost $g$, the aim is to find a control $\hat \gamma$ solving the following minimization problem
\begin{equation}
\inf_{\gamma} \E\left[ \int_0 ^T f(t,X_t,\mu_t ,\gamma_t) dN_t+g(X_T,\mu_T)\right]
\end{equation}\noindent
over all control processes $\gamma$ as above and such that the so-called mean field condition is fulfilled: the measure $\mu$ has to be equal to the law of the optimally controlled state variable.

This model is the natural limit of a symmetric nonzero-sum $n$-player game for $n \ge 1$ large enough, where each player controls the drift, the volatility as well as the jump sizes of her/his own private state and the states are coupled through their empirical distribution. In more detail, consider the following dynamics for the private state, $X^i$, of player $i=1,\ldots,n$:
\begin{equation}\label{X^i} dX^{i,n} _t =     b(t,X^{i,n} _t, \bar \mu_t ^n , \gamma_t ^{i,n} ) dt + \sigma(t,X^{i,n} _t, \bar \mu_t ^n , \gamma_t ^{i,n} ) dW^i _t + \beta(t,X_t ^{i,n} , \bar \mu^n _{t} ,\gamma^{i,n} _t) d\widetilde N^i _t,\quad t\in[0,T],\end{equation}
where $\gamma^i_t $ is the control of player $i$ at time $t$, $\bar \mu^n _t = (1/n)\sum_{j=1}^n \delta_{X_t ^{j,n}}$ is the empirical distribution of the vector $(X_t ^{1,n},\ldots, X_t ^{n,n})$ of the private states of all players, $(W^i)_{i \ge 1}$ is a sequence of independent multivariate Brownian motions and $(\widetilde N^i)_{i \ge 1}$ is a sequence of compensated Poisson processes with the same intensity $\lambda(t)$. The goal of each player $i$ is minimizing some objective function, given by
\begin{equation}
\E\left[ \int_0 ^T f(t,X^{i,n} _t, \bar \mu^n _t ,\gamma^{i,n} _t) dt +g(X^{i,n} _T, \bar \mu^n _T)\right]
\end{equation} \noindent
over her/his controls. This is a symmetric stochastic differential game, where the agents interact through the empirical mean of their private states, entering in both the drift and the jump component. According to mean field game theory, we expect that as the number of player gets larger and larger, the $n$-player game just described tends in some sense to the minimization problem in (\ref{Xrelax}), that is the solution to the latter would provide a good approximation of some Nash equilibrium in the $n$-player game. In the present paper, we focus on the existence of solutions for the limit MFG and we postpone to future research the other equally important issues of uniqueness of the MFG solution and approximation of Nash equilibria of the $n$-player game when $n$ is large. 

While the uncontrolled counter-part of MFG, that is particle systems and propagation of chaos for jump processes, has been thoroughly studied in the probabilistic literature (see, e.g., \cite{Gra,jourdain2007nonlinear} and the very recent preprint \cite{andreis2017mckean}), MFGs with jumps have not attracted much attention so far. In particular, MFGs for jump-diffusions have not been considered so far in the literature. Indeed, most of the existing articles focus on non-linear dynamics with continuous paths, with the exception of some papers such as \cite{gomes2013continuous}, \cite{hafayed2014mean}, \cite{kolokoltsov2011mean}, and the pre-print \cite{cecchin2017probabilistic}. The paper \cite{hafayed2014mean} deals with stochastic control of McKean-Vlasov type (see \cite{carmona2013control} for a comparison between MFG and McKean-Vlasov control), whereas \cite{kolokoltsov2011mean} uses methods based on potential theory and nonlinear Markov processes. The article \cite{gomes2013continuous} performs a detailed analysis of MFGs for continuous-time Markov chains, and the very recent \cite{cecchin2017probabilistic} studies MFGs with finitely many states via a probabilistic approach based on a Poisson-type approximation. Finally, we cite also the paper \cite{fu2017mean}, where MFGs with singular controls are considered for the first time, hence introducing the possibility of jumps in the state variables. 

The approach we use in this paper is based on weak formulation of stochastic controls, relaxed controls and martingale problems, which is very much inspired by Lacker \cite{La}. We extend to MFGs with jump-diffusive state variables, where the jump size as well as drift and volatility coefficients are controlled, the main results established by D. Lacker in \cite{La}. More in detail, our main contributions can be summarized as follows:

\begin{itemize}
\item Using the very powerful relaxed approach to MFGs introduced by Lacker \cite{La}, we establish abstract existence results for a class of MFGs with a multivariate state variable of jump-diffusion type, under suitable growth conditions on its coefficients and on the cost functional, hence including the case of unbounded coefficients and controls. The existence of Markovian solutions is also considered and sufficient conditions granting it are given. The issue of approximation of the $n$-player games by means of the MFG solution is considered in a companion paper \cite{benazzoli2017varepsilon}, while the one of uniqueness is postponed to future research.\smallskip

\item We complement the abstract existence results with an illiquid interbank market toy model, inspired by the systemic risk model proposed by Carmona and co-authors in \cite{CFS}. Within this model, using techniques based on forward/backward stochastic differential equations, we can compute explicitly the Nash equilibrium and see the convergence to the solution of the limit MFG. Moreover, we perform some numerical experiments showing the role of illiquidity in driving the evolution over time of the controls and the state variables.
\end{itemize}

The paper is structured as follows. Section \ref{sec:relax} provides the relaxed version of the MFG with controlled jumps together with all the assumptions on the coefficients, while in Section \ref{sec:existence} we state and prove the main existence results in the bounded case. In Section \ref{sec:existence-unbbd}, we extend the existence result to the unbounded case, i.e. coefficients and controls can be unbounded. Section \ref{sec:markov} identifies sufficient conditions guaranteeing the existence of a Markovian MFG solution.
In Section \ref{sec:ToyModel} we study Nash equilibria and their convergence towards the MFG solution in the illiquid interbank model; some numerical experiments are also performed.
Finally, the paper ends with an appendix, collecting most of the technical results used in the proofs of the main theorems.

\section{The relaxed MFG problem with controlled jumps}\label{sec:relax}


Following the approach in \cite{La}, we are going to study a relaxed version of the MFG briefly described in the introduction. Slightly more precisely, we will re-define the state variable and the controls on a suitable canonical space supporting all the randomness sources involved in the SDE with jumps above, so that the solution to the MFG will be identified with a probability measure $P$ on that space, that can be seen as the joint law of the pair state/control as in \cite{La}. Therefore, finding a relaxed solution to the MFG above will boil down to finding a fixed point for a suitably defined set-valued map. The rest of this section sets up the main assumptions on the state variable and the cost functions as well as the precise definition of relaxed mean field game with controlled jumps, while the issue of existence of solutions will be addressed in the next section.


\subsection{Notation}
We start with some notation that will be used throughout the whole paper. Let $\cD=\cD([0,T];\R^d )$ denote the set of all \cadlag functions with values in $\mathbb R^d$, i.e. the set of all right continuous with left limit functions $x\colon[0,T]\to \R^d$, endowed with the Skorokhod topology $J_1$. We will refer to \cite{Bill} for all the definitions and results on this topology when needed. Given some metric space $(S,d)$, $\mathcal P(S)$ denotes the set of all probability measures defined on the measurable space $(S,\mathcal B(S))$, where $\mathcal B(S)$ is the Borel $\sigma$-field of $S$. $\mathcal P(S)$ will be equipped with the topology induced by the weak convergence of measures. Moreover, $\mathcal P^p(S)$, for some $p \ge 1$, denotes the set all $P \in \mathcal P(S)$ such that $\int_S d(x,x_0)^p P(dx) < \infty$ for some (hence for all) $x_0 \in S$. In this paper, the set $\mathcal P^p (S)$ will always be equipped with the Wasserstein metric
\[
d_{W,p}(\mu,\nu) = \inf_{\pi\in\Pi(\mu,\nu)}\left(\int_{S\times S}d(x,y)^p \pi(dx,dy)\right)^\frac{1}{p}, \quad \mu, \nu\in\cP^p(S),
\]
where
\[
\Pi(\mu,\nu) = \left\{\pi\in\cP(S\times S):\pi\text{ has marginals } \mu,\nu\right\}.
\]
Finally for any measure $\mu \in \mathcal P^p (S)$ for $S$ being either $\R^d$ or $\cD$ we will use the notation
\begin{eqnarray*} | \mu |^p &=& \int_{\R^d} |x|^p \mu(dx),\\
\| \mu\|_t ^p &=& \int_{\cD} (|x|_t ^*)^p \mu(dx) , \quad |x|_t ^* :=\sup_{s\in [0,t]} |x(s)| \, . 
\end{eqnarray*}

Product spaces will always be endowed with the product $\sigma$-fields. Given any smooth enough scalar function $\psi\colon\R^n \to \R$, $D\psi$ will denote its gradient, while $D^2_x\psi$ its Hessian. Finally, for any matrix $M$ we denote $\tr[M]$ its trace.

\subsection{Assumptions}
In what follows we will make use of the following standing assumptions on the coefficients $b,\sigma, \beta$, the costs $f, g$ and the initial probability distribution $\chi$ of the state process $X$.

\begin{myassump}{A} Let $p' > p \ge 1$ be given real numbers.
\label{ass:A}
\begin{enumerate}[(\ref{ass:A}.1)]
\item \label{ass:A:chi} $\chi$ is a distribution in $\cP^{p'}(\R^d )$.\medskip
\item \label{ass:A:lambda} The intensity function $\lambda : [0,T] \to (0,\infty)$ is measurable and bounded by some constant $c_\lambda$.
\item \label{ass:A:cont} The coefficients
\[ (b,\sigma,\beta) : [0,T] \times \mathbb R^d \times \cP^p (\R) \times A \to \mathbb R^d \times \R^{d\times m} \times \mathbb R^{d}, \]
as well as the costs $f \colon[0,T]\times\R^d \times\cP^p (\R^d )\times A \to \R$ and $g \colon\R^d \times\cP^p (\R^d )\to \R$ are continuous functions in all their variables.
\medskip
\item \label{ass:A:growth:b} There exists a constant $c_1 > 0$ such that
for all $(t,\alpha)\in[0,T] \times A$, $x, y\in\R^d$, and $\mu,\nu\in\cP^p (\R^d )$ we have
\begin{gather*}
\abs{b(t,x,\mu,\alpha)-b(t,y,\nu,\alpha)}+\abs{\sigma(t,x,\mu,\alpha)-\sigma(t,y,\nu,\alpha)} + \abs{\beta(t,x,\mu,\alpha)-\beta(t,y,\nu,\alpha)} \le c_1(\abs{x-y} + d_{W,p}(\mu,\nu))\,,\\
\abs{b(t,x,\mu,\alpha)}+\abs{\sigma \sigma^{\top}(t,x,\mu,\alpha)} + \abs{\beta(t,x,\mu,\al)} \le c_1 \left( 1+ |x| + \left(\int_{\R^d} |z|^p \mu(dz)\right)^{1/p} + |\alpha| \right)\,.
\end{gather*}
\item \label{ass:a:growth:f:g} There exists some positive constants $c_2, c_3 >0$ such that for each $(t,x,\al) \in [0,T] \times \R^d \times A$ and $\mu,\nu \in \cP^p(\R^d)$ we have
\begin{gather*}
\abs{f(t,x,\mu,\alpha)-f(t,x,\nu,\alpha)} \le c_2 d_{W,p}(\mu,\nu)\,,\\
-c_2 \left( 1+ |x|^p + |\mu|^p \right) + c_3 |\alpha|^{p'}
 \le f(t,x,\mu,\al) \le c_2 \left( 1+ |x|^p + | \mu |^p + |\alpha|^{p'} \right) \,,\\
\abs{g(x,\mu)} \le c_2 \left( 1+ |x|^p + | \mu |^p  \right)\, .
\end{gather*}
Without loss of generality we can assume $c_1=c_2=c_\lambda$.
\medskip

\item \label{ass:A:A} The control space $A$ is a closed subset of $\R^q$ for some integer $q\ge 1$.
\end{enumerate}
\end{myassump}

Few comments on the assumptions above are in order. Conditions~(\ref{ass:A}.\ref{ass:A:lambda}), (\ref{ass:A}.\ref{ass:A:cont}) and (\ref{ass:A}.\ref{ass:A:growth:b}) ensure the existence of a unique strong solution of the SDE~\eqref{eqeqX} governing the evolution of the state variable.
The growth conditions postulated in Assumptions~(\ref{ass:A}.\ref{ass:A:growth:b}) and (\ref{ass:A}.\ref{ass:a:growth:f:g}) are widely used in Lemma~\ref{Jcont} and Lemma~\ref{Rcont}, which establish crucial compactness and continuity properties needed in the fixed point argument, hence the existence of a MFG solution when the space of actions $A$ is compact. In particular, we stress that (compared to \cite{La}) we need to impose an extra Lipschitz continuity of the coefficients as well as of the running cost with respect to the measure, while in \cite{La} only continuity is needed. This is due to the fact that we have to work with the Shorokhod space, where the flows of measures are not necessarily continuous in time (see the proofs of Lemmas \ref{Jcont} and \ref{Rcont}).

Moreover, they play an important role also when extending the existence of a MFG solution to the unbounded case. This will be done by truncating the coefficients first, and then passing to the limit along the sequence of MFG solutions corresponding to the truncated data. The growth conditions on the coefficients $b,\sigma, \beta$ and the costs $f,g$ allow to safely pass to the limit while preserving the optimality. 

\begin{remark}
Observe that we have chosen a volatility coefficient $\sigma$ with linear growth, i.e. $p_\sigma =1$ in \cite{La} notation, for the sake of simplicity. The two interesting cases covered by our assumptions essentially are $p=1,p'=2$ and $p=2,p'>2$, so that our toy model of illiquid interbank market of Section \ref{sec:ToyModel} is partially covered (see our Remark \ref{rmk:model-vs-theory}). More specifically the state variable therein fits the general theory, while the costs do not as they are quadratic in the control. We stress that even in the continuous paths framework in \cite{La}, quadratic costs are not allowed here in full generality (see Lacker's example in \cite[Section 7]{La}).
\end{remark}

\subsection{Relaxed controls}
\label{ssec:relMFG}

Now we give some background on relaxed controls. The use of relaxed controls has a long tradition in optimal control theory, they have the advantage of linearizing the objective functionals while having very convenient compactness properties. We refer to \cite{borkar2005controlled} and the references therein for an overview of such an approach.

Let $\Gamma$ be a measure on the set $[0,T] \times A$, equipped with the product $\sigma$-field $\mathcal B([0,T] \times A)$, such that its first marginal equals the Lebesgue measure, i.e. $\Gamma([s,t]\times A)=t-s$, for all $0 \le s \le t \le T$, and its second marginal is a probability distribution over $A$.

The set of all measures $\Gamma$ of this type and satisfying
\[ \int_{[0,T] \times A} |\al |^p \Gamma(dt,d\al) < \infty \]
will be denoted by $\V = \V[A]$ and each element $\Gamma\in\V[A]$ is called \emph{relaxed control}.
Since $\Gamma ([0,T]\times A) = T$ for all $\Gamma \in \V$, we can renormalize all such measures and endow $\V$  with the (modified) $p$-Wasserstein metric $d_{\V}$, given by
\begin{equation}\label{2-modWass}
d_\V(\Gamma,\Gamma')=d_{W,p}\left(\frac{\Gamma}{T},\frac{\Gamma'}{T}\right)\,.
\end{equation}
Notice that such a metric turns $\V$ into a complete separable metric space. Moreover, when the action space $A$ is compact so is $\V$.

Relaxed controls can be seen as a generalization of regular controls. Indeed, every relaxed control $\Gamma\in\V$ can be related to a unique (up to  a.e. equality) measure-valued map $t\mapsto \Gamma_t\in\cP (A)$ such that $\Gamma(dt,d\al)=dt\Gamma_t(d\al)$. Moreover a control $\Gamma\in\V$ is said to be \emph{strict} if $\Gamma_t=\delta_{\gamma(t)}$ for some $A$-valued measurable function $\gamma : [0,T] \to A$, where $\delta_x$ denotes the Dirac delta function at the point $x$.

\subsection{Admissible laws} Let $\cD=\cD([0,T];\mathbb R^d)$ and let $\mathcal{F}^\cD$ be the Borel $\sigma$-algebra induced on $\cD$ by the Skorokhod norm $J_1$. Moreover, we also consider the measurable space $(\V,\mathcal{B}(\V))$, where $\mathcal B(\V)$ is the Borel $\sigma$-field associated with the (modified) $p$-Wasserstein distance defined in (\ref{2-modWass}). Let $\Omega[A]=\V \times \cD$ be endowed with the product $\sigma$-field. A generic element $\Omega[A]$ is denoted by $\omega = (\Gamma,X)$, and, with a slight abuse of notation, we denote $\Gamma$ (resp. $X$) its projection onto $\V$ (resp. $\cD$). 

The product space $\Omega[A]$ will be equipped with the canonical filtration associated to $(\Gamma,X)$, which is defined as the product filtration $\mathcal F_t = \mathcal F_t ^X \otimes \mathcal F^\Gamma _t$ for $t \in [0,T]$, where $\mathcal{F}^X_t=\sigma\left(X_s, 0\le s\le t \right)$ and  $\mathcal{F}^\Gamma_t=\sigma(\Gamma(F): F\in\mathcal{B}([0,t]\times A))$, which are the canonical filtrations generated, respectively, by $X$ and by $\Gamma$. \medskip

Let $L$ be the linear integro-differential operator defined on $ C^\infty_0(\R^d)$, i.e. the set of all infinitely differentiable functions $\phi\colon\R^d \to\R$ with compact support, by 
\begin{eqnarray}
\label{opL}
L\phi(t,x,\mu, \al)&=&b(t,x,\mu, \alpha)^\top D\phi(x)+\frac{1}{2}\tr(\sigma \sigma^\top)(t,x,\mu,\alpha)D^2 \phi(x)\\
&&+ \left[\phi(x+\beta(t,x,\mu , \alpha))-\phi(x)- \beta(t,x,\mu ,\al)^\top D\phi (x)\right] \lambda(t) \nonumber
\end{eqnarray}
for each $(t,x,\mu,\al )\in[0,T]\times \R^d \times \cP^p (\R^d ) \times A$.

For each test function $\phi\in C^\infty_0(\R^d )$ and for any measure $\mu\in \cP^p (\cD)$, $\cM_t ^{\mu,\phi} \colon \Omega[A] \to \R$ is the operator defined by
\begin{equation}
\label{opM}
\cM^{\mu,\phi}_t(\Gamma,X)=\phi(X_t)-\int_{[0,t]\times A} L\phi(s,X_{s-}, \mu_{s-},\al) \Gamma_s (d\al) ds, \quad t \in [0,T],
\end{equation}
where $\mu_{t-}=\mu\circ \pi^{-1}_{t-}$ with $\pi_{t-}: \cD \to \R^d$ defined as $\pi_{t-}(x)=x_{t-}$ for $x \in \cD$. Notice that a.e. under the Lebesgue measure we have $\mu_{t-} = \mu_t$, where $\mu_t$ is defined similarly as the image of $\mu$ via the mapping $\pi_t : \cD \to \R^d$ given by $\pi_t (x)=x_t$, $x \in \cD$.

\begin{definition}
\label{defR}
Let $\mu$ be a given probability measure in $\cP^p (\cD)$. We say that $P \in \cP^p (\Omega[A])$ is an \emph{admissible law} if it satisfies the following conditions:
\begin{enumerate}
\item \label{condin}$P \circ X^{-1} _0 = \chi\in\cP^{p'} (\R)$;
\item \label{condGamma} $\E^P[ \int_0^T \abs{\Gamma_t}^p\,dt ] < \infty$;
\item \label{martingaleproperty} $\cM^{\mu,\phi}=(\cM^{\mu,\phi}_t)_{t\in[0,T]}$ is a $P$-martingale for each $\phi\in C^\infty_0 (\R^d)$.
\end{enumerate}
We denote by $\mathcal R (\mu)$ the set of all admissible laws.
\end{definition}

\begin{remark}\label{propR}
Observe that, according to Definition~\ref{defR} above, $\cR$ represents a set-valued correspondence $\cR \colon \cP^p (\cD)\corr\cP^p (\Omega[A])$. Moreover, for each probability distribution $\mu\in \cP^p (\cD)$, $\cR(\mu)$ is nonempty if the martingale problem \eqref{opM} (with initial distribution $\chi$) admits at least one solution. The latter is guaranteed by the fact that the SDE for $X$ has one strong solution by Assumption \ref{ass:A}. Finally, $\cR(\mu)$ is a convex set for each $\mu\in \cP^p (\cD)$, i.e. any convex combination $a P_1+(1- a)P_2$ with $a\in[0,1]$ and $P_1,P_2\in\cR(\mu)$, is still an element of $\cR(\mu)$. 
\end{remark}

We conclude this part on admissible laws with the following equivalent characterization of the elements of $\cR(\mu)$ as weak solutions to a suitable stochastic differential equation with jumps. For the notion of martingale measure and related stochastic integrals, which are used in the result below, we refer to the article \cite{EKM1990}. Similar techniques as therein lead to the following result. We provide a sketch of the proof below.

\begin{lemma}
\label{defR2} Let $\mu$ be a given probability measure in $\cP^p (\cD)$. Then $\cR(\mu)$ equals the set of all probability measures $Q$ on some filtered probability space
$(\Omega',\mathcal{F'},\F' ,Q)$, with $\F'=(\mathcal{F}'_t)_{t\in[0,T]}$
satisfying the usual conditions, supporting $m$ orthogonal $\F'$-adapted (continuous) martingale measures $M=(M^1, \ldots, M^m)$ on $[0,T] \times A$ with intensity $\Gamma_t (d\al)dt$, a random Poisson measure $\mathcal N$ on $[0,T]\times A$ with compensator $ \Gamma_t (d\al)\lambda(t) dt $,
and an $\F'$-adapted process $X$ satisfying the following equation
\begin{align}
\label{eqeqX}
dX_t= \int_A b(t,X_t,\mu_t, \al) \Gamma_t (d\al) dt+\int_A \sigma(t,X_t, \mu_t , \al) M(dt,d\al)+\int_{A} \beta(t, X_{t-},\mu_{t-}, \al) \widetilde{\mathcal N}(dt,d\al)\,,
\end{align}
with initial distribution $Q\circ X^{-1}_0=\chi$
and where $\widetilde{\mathcal N}$ denotes the compensated random measure, i.e. $\widetilde{\mathcal N}(dt,d\al)=\mathcal N(dt,d\al)-\Gamma_t(d\al)\lambda(t)dt$.
\end{lemma}

\begin{proof} Since any probability measures $Q$ satisfying the properties in the statement clearly belongs to $\cR(\mu)$, we only need to show the other inclusion. The latter follows as in \cite[Theorem IV-2]{EKM1990}, which is in turn based on the representation of continuous martingales as integrals with respect to martingale measures in \cite[Theorem III-10]{EKM1990}. Such a result can be extended to any c\`adl\`ag local martingale $Z$ by using the standard decomposition $Z=Z_0 + Z^c + Z^d$, where $Z^c$ is a continuous local martingale and $Z^d$ is a purely discontinuous one. Theorem III-10 in \cite{EKM1990} can be applied to the continuous part, while Th\'eor\`eme 13 in \cite{el1977poisson} can be used to represent $Z^d$ as integral with respect to some (compensated) Poisson random measure $\widetilde{\mathcal N} (dt,d\al)$ in a possibly enlarged probability space.     
\end{proof}

\begin{remark}
Observe that in Lemma~\ref{defR2}, which provides an equivalent definition of the admissible laws $\cR(\mu)$, the measurable space $(\Omega',\mathcal{F'})$ and the filtration $\F'$ are not specified in advance. However, by definition, $\Gamma$ is an element in $\V$ and the solution process $X$ has \cadlag path. Therefore, by considering the measurable map
\[ \Omega' \ni \omega \mapsto(\Gamma(\omega), X(\omega))\in \Omega[A] = \V \times \cD
\]
we can induce a measure $P'$ on the canonical space. Furthermore, $(\Gamma,X)$ has the same law under $P'$ as it does under $P$. Therefore we will assume that $P\in\cR(\mu)$ means that $P$ is defined on the canonical space.
\end{remark}
\begin{remark} Note that under Assumption \ref{ass:A}, in particular Lipschitz continuity and growth conditions on the coefficients $b,\sigma, \beta$, the SDE with jumps \eqref{eqeqX} admits a unique strong solution.  
\end{remark}

\subsection{Relaxed mean field game} For any probability measure $\mu \in \cP^p (D)$, the objective cost functional of the  minimization problem is $\cC^\mu\colon\Omega[A]\to\R$ defined by 
\begin{equation}
\label{minCrelaxed}
\cC^\mu(\Gamma,X)=\int_{[0,T]\times A}f(t,X_t,\mu_t,\al)\Gamma(dt,d\al)+g(X_T,\mu_T)\,.
\end{equation}
Thus, solving the relaxed MFG with controlled jump component means finding a probability measure $P^* \in \cR(\mu)$ so that the expected cost under $P^*$ is minimal with respect to all admissible laws, i.e. 
\[
\int_{\Omega[A]}\cC^\mu\,dP^*=\inf_{P\in\cR(\mu)} \int_{\Omega[A]} \cC^\mu \,dP 
\]
together a suitable \emph{mean field condition}. In order to give a rigorous formulation of the latter and hence complete the description of relaxed MFG, we need to introduce the set-valued map $\cR^*\colon \mathcal P^p (\cD) \corr\cP(\Omega[A])$ defined as \medskip
\begin{equation}
\label{corrR}
\mathcal P^p (\cD) \ni \mu \mapsto \cR^*(\mu)=\arg\min_{P\in\cR(\mu)} J(\mu,P) \subset \cP(\Omega[A]) \,,
\end{equation}
where $J\colon P(\cD) \times \cP(\Omega[A]) \to \R\cup\{\infty\}$ is the expected cost given by
\begin{equation}
\label{opJ}
J(\mu,P) =\E^P\left[\cC^\mu\right]=\int_{\Omega[A]}\cC^\mu\,dP\,.
\end{equation}

\begin{remark} Observe that $\cR^* (\mu) \subset \cP^p (\Omega[A])$ whenever $\mu \in \cP^p (\cD)$. Indeed, by definition of the set $\mathcal R$, any $P\in\cR$ satisfies $\E[\int_0^T \abs{\Gamma_t}^p\,dt ] < \infty$, hence $\E^P\left[( \abs{X}^*_T)^p\right] < \infty$ in view of Lemma~\ref{stime}. Therefore $P\in\cP^p(\Omega[A])$ and being $\cR^*\subset\cR$ the conclusion holds. \end{remark}

We can at last give the rigorous definition of relaxed MFG solutions in a setting where the jump component is controlled as well.
\begin{definition}\label{def:rMFGsol} A  \emph{relaxed MFG solution} is a probability distribution $P\in\cP^p (\Omega[A])$ such that $P\in\cR^*(P\circ X^{-1})$, i.e. it provides a fixed point for the set-valued map
\[
\mathcal P^p (\cD)\ni \mu \mapsto \{P\circ X^{-1} : P\in\cR^*(\mu)\}\,.
\]
A relaxed MFG solution is said to be \emph{Markovian} (resp. \emph{strict Markovian}) if the $\V$-marginal of $P$, i.e. $\Gamma$, satisfies $P(\Gamma(dt,d\al)=dt\hat \Gamma(t,X_{t-})(d\al))=1$ for a measurable function $\hat \Gamma\colon[0,T]\times\R\to\cP^p (A)$ (resp. $P(\Gamma(dt,d\al)=dt\delta_{\hat\gamma(t,X_{t-})}(d\al))=1$ for a measurable function $\hat \gamma\colon[0,T]\times\R\to A$).
\end{definition}
The next sections address the issues of existence of relaxed MFG solutions in both cases for bounded and unbounded coefficients as well as the existence of Markovian solutions.

\section{Existence of a relaxed MFG solution in the bounded case}\label{sec:existence}

In this section we prove the existence of a relaxed solution of our MFG with controlled jump-diffusion dynamics under the additional assumptions of boundedness of the coefficients and compactness of the action space $A$. Moreover, we will also see that, under some rather standard convexity property, one can provide a (strict) Markovian MFG solution. The technical results that we use in the proofs can be found in the Appendix~\ref{sec:intermediateresults}. The following assumption will be in force throughout the whole section.

\begin{myassump}{B} \label{ass:B}
The coefficients $b,\sigma, \beta$ are bounded and the action space $A$ is compact.
\end{myassump}

\begin{theorem}
\label{thm:relaxedMFGs}
Under Assumptions \ref{ass:A} and \ref{ass:B}, there exists a relaxed MFG solution.
\end{theorem}
\begin{proof}
According to Definition~\ref{def:rMFGsol}, a probability distribution $P\in\cP^p (\Omega[A])$ is a relaxed MFG solution if it is a fixed point for the correspondence $\cE\colon \cP^p (\cD) \corr \cP^p (\cD)$ given by \begin{equation}
\label{corrE}
\cP^p (\cD) \ni \mu \mapsto \cE(\mu) :=\set{P\circ X^{-1} : P\in\cR^*(\mu)}\notag ,
\end{equation}
where $\cR^*$ denotes the correspondence defined in \eqref{corrR}.
In order to prove the existence of such a $P$, we apply the Kakutani-Fan-Glicksberg fixed point theorem (see, e.g., \cite[Theorem 17.55]{AlBo}) to a restriction of $\cE$ to a suitably chosen domain. Indeed, that theorem applies to convex-values correspondences with closed graph and nonempty, compact, convex domain. Therefore we look for a convex compact subset $\mathcal K \subset \cP^p (\cD)$, containing $\cE(\cP^p (\cD))$, and consider the restriction of $\cE$ on $\mathcal K$, which we will denote by $\cE_{\mathcal K} \colon \mathcal K \corr \mathcal K$. We split the remainder of the proof into four steps. \medskip

\emph{Step 1: the map $\cE$ is upper hemicontinuous with non-empty compact convex values (hence it has closed graph by \cite[Theorem 17.11]{AlBo}).}

Lemma~\ref{Jcont} implies the joint continuity of the function $J$ defined in \eqref{opJ}, whereas Lemma~\ref{Rcont} yields that $\cR$ is continuous and has nonempty compact values, and therefore by applying the Berge Maximum Theorem (see \cite[Theorem 17.31]{AlBo}), we have that the correspondence $\cR^*$ is indeed upper hemicontinuous with nonempty compact values. By continuity of the map $\cP^p (\Omega[A])\ni P\mapsto P\circ X^{-1} \in \cP^p (\cD)$ so is $\cE$ (cf. \cite[Theorem 17.23]{AlBo}).
Recall that, by Remark~\ref{propR}, $\cR(\mu)$ is convex for each $\mu$. Hence by linearity\footnote{In this context, by linearity of the map $P \mapsto P\circ X^{-1}$ we mean linearity with respect to convex combinations, namely:
$(a P + (1-a) Q)\circ X^{-1} (B) =( a P \circ X^{-1} + (1-a) Q\circ X^{-1}) (B)$ for any Borel set $B$ and for all $a \in [0,1]$. Also $J$, being an expected value, is linear with respect to convex combinations in the underlying probability $P$.} and continuity of $P\mapsto J(\mu,P)$ and of $P \mapsto P\circ X^{-1}$, it follows that also $\cR^*(\mu)$ and $\cE(\mu)$ are convex sets for each $\mu\in \cP^p (\cD)$.\medskip

\emph{Step 2: identification of a suitable auxiliary compact convex subset of $\cP^p (\Omega[A])$.}

Let $\chi \in \mathcal P^{p'}(\mathbb R^d )$ be a given initial law. Let $\cQ$ be a set of the probability measures $P$ in $\cP^p (\Omega[A])$, such that:
\begin{enumerate}[(i)]
\item \label{cond:cQ1} $X_0\sim\chi$;
\item \label{cond:cQ2} $\E^P \left[ (\abs{X}^*_T )^{p} \right] \le C$, where $C = C(T,c_1,c_\lambda,\chi)$ denotes the constant appearing in equation~\eqref{bound-X} of Lemma~\ref{stime}, which is independent of $P$;
\item \label{cond:cQ3} $X$ is adapted to a filtration $\F=(\mathcal F _t)_{t\in[0,T]}$ and satisfies
\begin{equation}\label{cond:cQZ}
\E^P \left[\left(X_{(t+u)\wedge T}-X_t\right)^p \vert\mathcal{F}_t \right]\le C_1 \delta
\end{equation}
for $t\in[0,T]$ and $u\in[0,\delta]$, with $C_1=\bar C(c_1,c_\lambda,\sigma)$ defined in \eqref{def:barC} independently of $P$.
\end{enumerate}
We need to prove that $\cQ$ is a nonempty, compact and convex subset of $\cP^p (\Omega[A])$. 
First, $\cQ$ is convex by construction: consider $P^\prime = a P_1 +(1-a) P_2$ with $a\in[0,1]$, $P_1,P_2\in\cQ$ and corresponding filtrations $\mathbb F^1$ and $\mathbb F^2$ as in condition (iii) above. Conditions~\eqref{cond:cQ1} and \eqref{cond:cQ2} are easily satisfied by $P^\prime$ since the initial distribution $\chi$ is the same for all the probabilities and the constant $C$ does not depend on them.
Condition~\eqref{cond:cQ3} for $P^\prime$ also holds with the same constant $C_1$ as in \eqref{cond:cQZ} and the filtration $\mathbb F^\prime = \mathbb F^1 \wedge \mathbb F^2$.

Now, we show that $\cQ$ is relatively compact in $\cP^p(\Omega[A])$. Observe that $\cQ$ is tight since it satisfies the sufficient criterion for tightness \cite[Lemma 3.11]{Whitt} \footnote{Notice that even though Whitt's result is stated only for $p=2$, a careful inspection reveals that it holds for any $p>0$ as well since it relies on his Theorem 3.3 where the exponent is any $p>0$.}. Indeed since the constant $C_1$ in~\eqref{cond:cQZ} is independent of $P$, it suffices to choose (in the notation of \cite[Lemma 3.11]{Whitt}) $Z(\delta)=C_1\delta$, and the tightness follows. Applying \cite[Prop. 7.1.5]{ambrosio} and the bound (ii) in the definition of $\cQ$ we get that $\cQ$ is bounded in $\cP^p (\Omega[A])$. Therefore $\cQ$ is relatively compact, hence its closure $\overline \cQ$ for the $p$-Wasserstein metric is compact in $\mathcal P^p (\Omega[A])$ and convex.
\medskip

\emph{Step 3: identification of a compact subset $\mathcal K$ of $\cP^p (\cD)$.}

We can now define $\mathcal K$ as follows
\[
\mathcal K=\set{\nu \in \cP^p (\cD) : \text{there exists $P\in \overline \cQ$ such that $\nu=P\circ X^{-1}$}}.
\]\noindent
Since $\overline \cQ$ is compact and convex and $P\mapsto P\circ  X^{-1}$ is a continuous and linear map, $\mathcal K$  turns out to be a convex, compact subset of $\cP^p (\cD)$ as requested in Kakutani-Fan-Glicksberg fixed point theorem. \medskip

\emph{Step 4: prove that $\cE_{\mathcal K} (\mathcal K) \subset \mathcal K$, i.e. $\cE_{\mathcal K}(\mu) \subset \mathcal K$ for all $\mu \in \mathcal K$.} 

In order to show that the range of $\cE_{\mathcal K}$ is contained in $\mathcal K$ we prove that  $\cR(\mu)\subset\cQ$ for each $\mu\in \cP^p (\cD)$, so that $\cE(\mu)\subset\mathcal K$ for all $\mu\in \cP^p (\cD)$ and therefore $\cE_{\mathcal K}(\mathcal K)\subset \mathcal K$. Let $P\in\cR(\mu)$, then it satisfies conditions~\eqref{cond:cQ1} and \eqref{cond:cQ2} by construction (see Definition~\ref{defR2} and  Lemma~\ref{stime}). The validity of condition~\eqref{cond:cQ3} can be proved arguing as in Proposition~\ref{Qccompact}. Indeed, using the same notation therein, we have that for each $u\in[0,\delta]$
\[
\E^P\left[\left(X_{(t+u)\wedge T}-X_t\right)^p \mid \mathcal F_t \right]\le \bar C(c_1,c_\lambda,\sigma) u\le \bar C(c_1,c_\lambda,\sigma)\delta \,,
\]
giving the same bound as in \eqref{cond:cQZ} with constant $\bar C(c_1,c_\lambda,\sigma)$, which does not depend on $P$. 
Note that since $\cR(\mu)$ is nonempty (see Remark~\ref{propR}) then so is $\cQ$.\medskip

Since all the hypotheses of the Kakutani-Fan-Glicksberg fixed point theorem are satisfied, we can conclude that there exists a fixed point for {the correspondence $\cE_{\mathcal K}$, that is a fixed point also for $\cE$, which is by definition a relaxed MFG solution.}
\end{proof}

\section{Existence of a relaxed MFG solution in the unbounded case}\label{sec:existence-unbbd}
By adapting to our setting the arguments used in \cite[Section 5]{La}, we extend the existence result in the previous section to the case when $b$, $\sigma$ and $\beta$ have linear growth and the action space $A$ is not necessarily compact. The main idea is to work with an approximation of these functions, namely their truncated and therefore bounded versions. Then, by a convergence argument, it must be shown that the limit of the mean field game solutions found in the truncated setting is indeed a solution for the unbounded case. The main result can be easily formulated as follows:
\begin{theorem}\label{thm-unbdd}
Under Assumption \ref{ass:A}, there exists a relaxed MFG solution.
\end{theorem}
Our proof follows closely \cite[Section 5]{La}, hence instead of giving all details, which would be rather redundant, we will sketch the main arguments and we will give more details only on those parts which are jump-specific.
  
In preparation for the proof, let us introduce some notation. For any $n \ge 1$, let $(b_n,\sigma_n,\beta_n)$ be the truncated version of the coefficients $(b,\sigma, \beta)$, i.e. $b_n$ is the pointwise projection of $b$ into the ball centered at the origin with radius $n$ in $\mathbb R^d$ and analogously for $\sigma_n$ and $\beta_n$ in their respective value spaces. Moreover, we denote $A_n$ the intersection of $A$ with the ball centered at the origin with radius $r_n = \sqrt{n/2c_1}$, where we recall that $c_1$ is the constant appearing in Assumption (A.\ref{ass:A:growth:b}) granting Lipschitz continuity as well as growth conditions on the coefficients of the state variable. Since $A$ is closed by assumption, there exists $n_0$ such that for all $n \ge n_0$ the set $A_n$ is nonempty and compact, hence the truncated data set $(b_n,\sigma_n,\beta_n, f,g, A_n)$ satisfies Assumptions \ref{ass:A} and \ref{ass:B} and, furthermore, properties (\ref{ass:A}.4) and (\ref{ass:A}.5) are fulfilled with the same constants $c_i$ ($i=1,2,3$) independent of $n$.

Due to Theorem \ref{thm:relaxedMFGs} in the previous section, for all $n \ge 1$ there exists a relaxed MFG solution corresponding to the data set $(b_n,\sigma_n,\beta_n,f,g,A_n)$, which can be viewed as a probability measure on $\Omega[A]$ since $\cP(\Omega[A_n])$ can be naturally embedded in $\cP(\Omega[A])$ due to the inclusion $A_n \subset A$. Let $L^n$ be the operator defined as $L$ in (\ref{opL}) with the truncated data $(b_n,\sigma_n,\beta_n)$ replacing $(b,\sigma,\beta)$. Now, we can define the set of admissible laws $\mathcal R_n (\mu)$ as the set of all measures $P\in \cP(\Omega[A])$ such that\begin{enumerate}
\item $P(\Gamma([0,T] \times A_n ^c)=0)=1$;
\item $P\circ X_0 ^{-1}=\chi$;
\item for all functions $\phi \in C_0 ^\infty (\R^d)$, the process 
\[ \cM_t ^{\mu,\phi ,n} := \phi(X_t)-\int_{[0,t]\times A} L^n \phi(s,X_{s-}, \mu_{s-},\al) \Gamma_s (d\al) ds, \quad t \in [0,T], \]
is a $P$-martingale. 
\end{enumerate}
We also define $\cR^* _n (\mu) := \arg \max_{P\in \cR_n (\mu)} J(\mu,P)$. Due to the embedding of $\cP(\Omega[A_n])$ in $\cP(\Omega[A])$, we can identify $\cR_n (\mu)$ (resp. $\cR_n ^* (\mu)$) with the set of admissible laws (resp. optimal laws) of the MFG with data $(b_n,\sigma_n,\beta_n, f,g,A_n)$. Finally, any relaxed MFG solution for the $n$-truncated data can be viewed as a probability $P_n \in \cR^* _n (\mu^n)$ with $\mu^n \in \cP^p (\cD)$ satisfying the mean field condition $\mu^n = P_n \circ X^{-1}$. We are now ready to give the proof of Theorem \ref{thm-unbdd}. 

\begin{proof}[Proof of Theorem \ref{thm-unbdd}] The proof is structured in three steps.\medskip 

\noindent\emph{Step 1: the sequence of relaxed MFG solutions $(P_n)_{n \ge 1}$ is relatively compact in $\cP^p(\Omega[A])$. Furthermore, we have
\begin{equation} \label{uniform-bounds} \sup_n \E^{P_n} \left[ \int_0 ^T |\Gamma_t | ^{p'} dt \right] <\infty, \quad \sup_n \E^{P_n} \left[(|X|_T ^*)^{p'} \right] = \sup_n \| \mu ^n \|_T ^{p'} < \infty.  \end{equation} }
To prove the two uniform bounds in (\ref{uniform-bounds}) one proceeds verbatim as in \cite[Lemma 5.1]{La}. Regarding the relative compactness of the sequence $(P_n)_{n \ge 1}$, it follows from Proposition \ref{Qccompact}.\medskip

\noindent\emph{Step 2: For any $P \in \cP^p(\Omega[A])$, limit in $\cP^p (\Omega[A])$ of a convergent subsequence $(P_{n_k})_{k \ge 1}$ of $(P_n)_{n\ge 1}$, we have $P \in \cR (\mu)$, where $\mu=P \circ X^{-1}$, and
\begin{equation} \label{u-bound-limit} \E^P \left[ \int_0 ^T |\Gamma_t|^{p'} dt \right] < \infty.\end{equation}} 
The proof of this step is very close to the one of \cite[Lemma 5.2]{La} except for some additional details due to the presence of jumps in the state variable. We sketch the main argument by stressing the few differences with Lacker's proof. First of all, we have $\mu=\lim_k \mu^{n_k} = \lim_k P_{n_k}\circ X^{-1}=P\circ X^{-1}$, while the bound (\ref{u-bound-limit}) is a consequence of the LHS in (\ref{uniform-bounds}) and a standard application of Fatou's lemma. 

To conclude this step it suffices to show that $P$ is an admissible law for the non-truncated MFG, i.e. $P \in \cR(\mu)$. This boils down to showing that $\cM^{\mu,\phi}$ is a $P$-martingale for all $\phi \in C^\infty _0 (\R^d)$ relying in turn on the fact that for all $n \ge 1$ the process $\cM^{\mu^n, \phi, n}$ is a $P_n$-martingale due to $P_n \in \cR_n (\mu^n)$. First, notice that for all $t \in [0,T]$
\begin{eqnarray*}
\cM^{\mu^n, \phi,n}_t (\Gamma,X) - \cM^{\mu^n ,\phi}_t (\Gamma,X) &=& \int_{[0,t] \times A} ds \Gamma_s (d\alpha)(b_n-b)^\top (s,X_s,\mu_s ^n , \alpha) D\phi(X_s) \\
&& \quad + \frac{1}{2}\tr{\left[\left(\sigma_n \sigma_n ^\top - \sigma \sigma^\top\right)(s,X_s,\mu_s ^n ,\alpha) D^2 \phi(X_s)\right]}  \\
&& \quad + \left[ \phi(X_s+\beta_n (s,X_s,\mu_s ^n ,\alpha))-\phi(X_s+\beta (s,X_s,\mu_s ^n ,\alpha)) \right .\\ 
&& \quad \left . - (\beta_n -\beta)^\top (s,X_s,\mu_s ^n ,\alpha)D\phi(X_s)\right] \lambda(s).
\end{eqnarray*} 
Proceeding exactly as in the proof of \cite[Lemma 5.2]{La}, we obtain the following bounds for the terms containing $(b_n -b)^\top$:
\begin{equation}
 \int_{[0,t] \times A} ds \Gamma_s (d\alpha) | (b_n-b)^\top (s,X_s,\mu_s ^n , \alpha) D\phi(X_s) | \le 2Cc_1 \left( tZ_1 + \int_0 ^t |\Gamma_s| ds \right) \mathbf 1_{\{2c_1 Z_1 > n\}}
\end{equation}
where $C$ is a constant such that $|\phi(x)|+|D\phi(x)| + |D^2 \phi(x)| \le C$ for all $x \in \R^d$, $n$ is large enough (more precisely, for $n \ge 2c_1$), and
\begin{equation}\label{Z1} Z_1 := 1+ |X|^* _T + \left( \sup_{n \ge 1} \int_{\cD} (|z|^* _T)^p \mu^n (dz)\right)^{1/p}.\end{equation}
The same bound holds for the term with $\sigma_n \sigma_n^\top - \sigma \sigma^\top$ as well (recall that in our case $p_\sigma =1$). Regarding the term coming from the jump part, we have the following estimates
\begin{eqnarray*}
\left | \phi(X_s+\beta_n (s,X_s,\mu_s ^n ,\alpha))-\phi(X_s+\beta (s,X_s,\mu_s ^n ,\alpha))
 - (\beta_n -\beta)^\top (s,X_s,\mu_s ^n ,\alpha)D\phi(X_s) \right |  \\ 
 \le \left | \phi(X_s+\beta_n (s,X_s,\mu_s ^n ,\alpha))-\phi(X_s+\beta (s,X_s,\mu_s ^n ,\alpha))\right|
 + \left | (\beta_n -\beta)^\top (s,X_s,\mu_s ^n ,\alpha)D\phi(X_s) \right | \\
 \le C \left(2 +  \left | (\beta_n -\beta)^\top (s,X_s,\mu_s ^n ,\alpha) \right | \right) \mathbf 1_{\{ | \beta (s,X_s,\mu_s ^n , \al)| > n\}},
\end{eqnarray*}
where we used the fact that $\beta_n - \beta \neq 0$ only if $|\beta_n | > n$. In a similar way as for $(b_n -b)^\top$ in \cite[Lemma 5.2]{La}, one can show that $|\beta_n (s,X_s,\mu_s ^n , \al)| > n$ implies $c_1 (Z_1 + |\al |) > n$ due to Assumption (A.\ref{ass:A:growth:b}) and by definition of $Z_1$ in \eqref{Z1}. Since $| (\beta_n - \beta)(s,X_s,\mu_s ^n , \al)| \le 2c_1 (Z_1 + |\al |)$, taking $n \ge 2c_1$ yields that $P_n$-a.s.
\begin{eqnarray*} \int_{[0,t] \times A} ds \Gamma_s (d\al) C \left(2 +  \left | (\beta_n -\beta)^\top (s,X_s,\mu_s ^n ,\alpha) \right | \right) \mathbf 1_{\{ | \beta (s,X_s,\mu_s ^n , \al)| > n\}} \\
\le 2 C \int_{[0,t] \times A} ds \Gamma_s (d\al) \left(1 +  c_1 ( Z_1 + |\al |) \right) \mathbf 1_{\{ c_1 (Z_1 + |\al |) > n\}}\\
\le 2  C \tilde c_1 \int_{[0,t] \times A} ds \Gamma_s (d\al) \left(1 +  Z_1 + |\al |) \right) \mathbf 1_{\{ c_1 Z_1  > n\}}\\
\le 2 C \tilde c_1 \left ( t (1+Z_1) + \int_0 ^t | \Gamma_s | ds \right) \mathbf 1_{\{ c_1 Z_1  > n\}},
 \end{eqnarray*}
 where we set $\tilde c_1 = c_1 \vee 1$. Therefore, combining the bounds above we obtain for all $t \in [0.T]$ and $P_n$-a.s.
 \[
\left | \cM^{\mu^n, \phi,n}_t (\Gamma,X) - \cM^{\mu^n ,\phi}_t (\Gamma,X) \right |  \le  6 C \tilde c_1 \left( t(1+Z_1) + \int_0 ^t |\Gamma_s| ds \right) \mathbf 1_{\{2c_1 Z_1 > n\}}.
\]
The estimate above, together with the ones in Step 1, implies the convergence
\[ \mathbb E^{P_n} \left[ \left | \cM^{\mu^n, \phi,n}_t (\Gamma,X) - \cM^{\mu^n ,\phi}_t (\Gamma,X) \right | \right] \to 0, \quad n \to \infty.\]
Finally, using the continuity of $\cM^{\nu ,\phi}_t (\Gamma,X)$ in $(\nu,\Gamma,X)$, granted by Lemmas \ref{Lemma:Cont} and \ref{Lemma:cont-mu}, we can conclude this step as in the proof of \cite[Lemma 5.2]{La}.\medskip

\noindent\emph{Step 3: The limit point $P$ is optimal, i.e. $P \in \cR^* (\mu)$.} First of all, let $P' \in \cR(\mu)$ with $J(\mu,P') < \infty$. Then, one can show that there exists a sequence of probabilities $P'_n \in \cR_n (\mu^n)$ such that
\begin{equation}\label{subseq} J_{n_k} (\mu^{n_k}, P'_{n_k}) \to J(\mu,P'), \quad k \to \infty,\end{equation}
where $J_n$ denotes the objective corresponding to the truncated data. Indeed, this can be shown as in the proof of \cite[Lemma 5.3]{La} with the difference that estimates for the jump part are needed as well. They can be obtained in a standard way using Burkholder-Davis-Gundy inequalities and the Assumption (\ref{ass:A}.\ref{ass:A:growth:b}). Therefore, since $P_n$ is optimal for each $n$ we have
\[ J_n (\mu^n , P_n ') \ge J_n (\mu^n ,P_n),\]
so that thanks to \eqref{subseq} and to the fact that $J$ is lower semicontinuous (cf. Lemma \ref{Jcont}) we obtain
\[ J(\mu,P) \le \lim \inf_{k \to \infty} J_{n_k}(\mu^{n_k}, P_{n_k}) \le \lim_{k \to \infty} J_{n_k}(\mu^{n_k}, P' _{n_k}) = J(\mu,P').\] The optimality of $P$ follows since $P'$ is arbitrary. \medskip

The proof is completed since we have exhibited an admissible law $P$, which is optimal (Step 3) and satisfies the mean field condition (Step 2). Hence $P$ is a relaxed MFG solution.
\end{proof}

\section{Existence of a Markovian MFG solution}\label{sec:markov}

The following theorem extends to our setting the results in \cite[Theorem 3.7]{La}, showing that for any element $P \in \cP^p (\mu)$, for $\mu \in \cP^p (\cD)$, there exists a Markovian control with a lower cost than $P$. We need to introduce one further assumption:
\begin{myassump}{C} \label{ass:C}
For all $(t,x,\mu)\in [0,T]\times \mathbb R^d \times \cP^p (\mathbb R^d )$, the subset
\begin{equation} \label{eq:setK}
K(t,x,\mu) := \left\{(b(t,x,\mu, \al), \beta(t,x,\mu ,\al) , \sigma \sigma^\top (t,x,\mu ,\al),z) : \al \in A, z \ge f(t,x, \mu ,\al)\right\} 
\end{equation}
of $\mathbb R^d \times \R^d \times \R^{d\times d} \times \R $ is convex. 
\end{myassump}

\begin{remark}
The additional assumption is a classical one in the relaxed approach to optimal control. We just notice that the toy model in Section \ref{sec:ToyModel} fulfils it. Slightly more generally, any model where $A$ is convex subset of $\mathbb R^d$, $b$ is affine in $(x, \al)$ and independent of $\mu$, $\beta (t,x,\mu,\al)=\beta_1\al + \beta_2$ (where $\beta_i$ have constant entries), $\sigma$ is a constant matrix and $f(t,x,\mu,\al)$ is convex in $\al$, satisfies Assumption \ref{ass:C}.  
\end{remark}

Now, we can state and prove the main result of this section.

\begin{theorem}
\label{thm:MarkovianMFGsol}
Under Assumption \ref{ass:A}, let $P$ be a relaxed MFG solution. Assume that the coefficients $(b,\sigma,\beta)$ satisfy the following property: there exist some constant $c>0$ and some continuous function $\psi : A \to (0,\infty)$ such that
\begin{equation}\label{hyp-kurtzstock} |b(t,x,\mu,\al)| + |\sigma \sigma^\top (t,x,\mu,\al)| + |\beta(t,x,\mu,\al)| \le c(1+ \psi(\al)),\end{equation}
for all $(t,x,\mu, \al) \in [0,T] \times \R^d \times \mathcal P^p (\R^d) \times A$. 
Then there exists a Markovian MFG solution.
Moreover, if also Assumption \ref{ass:C} holds, there exists a strict Markovian MFG solution.

\end{theorem}
\begin{proof}  
Let $P$ be a relaxed MFG solution. In order to show the existence of a Markovian MFG solution, first we exhibit a (possibly different) probability measure $P^*\in\cR(\mu)$, with $\mu := P \circ X^{-1}$, such that the following holds:
\begin{myprop}{M}\label{prop:M}
\begin{enumerate}[(\ref{prop:M}.1)]
\item \label{ott}$J(\mu,P^*)\le J(\mu,P)$;
\item \label{pfcond}$P^* \circ X^{-1}_t = P\circ X^{-1}_t$ for all $t\in[0,T]$;
\item \label{Markoviancond}$P^*  (\Gamma(dt,d\al)= \widehat \Gamma (t,X_{t-})(d\al)dt)=1$ for a measurable function $\widehat \Gamma \colon[0,T]\times \R^d \to \mathcal P(A)$.
\end{enumerate}
\end{myprop}
Condition (\ref{prop:M}.\ref{ott}) implies that under this new probability measure $P^*$ the expected cost is minimized with respect to all the admissible laws, i.e. $P^*\in\cR^*(\mu)$, whereas (\ref{prop:M}.\ref{pfcond}) assures that the marginals of $X$ are preserved under the new probability $P^*$. Condition (\ref{prop:M}.\ref{Markoviancond}) guarantees the Markovianity of $P^*$.

To obtain a Markovian MFG solution, we apply the mimicking results in \cite[Corollary 4.9]{kurtz1998existence}. In order to do that, we  need to check that $L$ satisfies properties (i)-(vi) in \cite[pp. 611-612]{kurtz1998existence}.\footnote{The results in \cite{kurtz1998existence} are established for a time-homogeneous generator. However, results in Section 4.2 of that article can be applied to our setting by suitably choosing the space of controls in \cite{kurtz1998existence} as $U=[0,T]\times \mathcal P^p (\R^d) \times A$, the relaxed control being as $\Lambda_t (du) = \delta_{(t,\mu_t)}(du^1,du^2) \Gamma_t (du^3)$, for all $u=(u^1,u^2,u^3)\in U$.}  Let us note that assumptions (i)-(iii) are easily checked, while our assumption \eqref{hyp-kurtzstock} readily implies (v). It remains to show that (iv) (positive maximum principle, cf. \cite{EK}, page 165) is also fulfilled. Let $\phi\in C_0 ^\infty(\mathbb R^d)$ and $x_0\in \mathbb R^d$ be such that $\sup_{x\in \mathbb R^d}\phi(x)=\phi(x_0) \ge 0$, then we have
\begin{align*}
L\phi(t,x_0 ,\mu , \al)&= b(t ,x_0 ,\mu , \alpha)^\top D \phi(x_0 ) +\frac{1}{2}\tr(\sigma \sigma^\top)(t ,x_0 ,\mu ,\alpha)D^2 \phi(x_0)\\
& \quad + \left[\phi(x_0 +\beta(t ,x_0 ,\mu , \alpha))-\phi(x_0 )- \beta(t, x_0 ,\mu ,\al)^\top D \phi (x_0)\right] \lambda(t) \\
& =  \frac{1}{2}\tr(\sigma \sigma^\top)(t ,x_0 ,\mu ,\alpha)D^2 \phi(x_0) + \left[\phi(x_0 +\beta(t ,x_0 ,\mu , \alpha))-\phi(x_0 )\right] \lambda(t),  \nonumber
\end{align*}
which is negative for all $(t,\mu,\al) \in [0,T]\times \mathcal P^p (\R^d) \times A$ since $\lambda (t) > 0$ for all $t \in [0,T]$. Therefore, we can apply the very general results established in \cite[Corollary 4.9]{kurtz1998existence} granting the existence of a process $Y$, defined on some filtered probability space $(\widehat \Omega, \widehat{\mathbb F}, \widehat P)$, and of a measurable function $\widehat \Gamma : [0,T]\times \R^d \to \mathcal P(A)$ such that
\[ \cM^{\mu,\phi}_t(\widehat \Gamma ,Y)=\phi(Y_t)-\int_{[0,t]\times A} L\phi(s,Y_s, \mu_{s},\al) \widehat \Gamma (s,Y_s, d\al) ds, \quad t \in [0,T], \]
is an $\widehat{\mathbb F}$-adapted $\widehat P$-martingale for all $\phi \in C^\infty _0 (\R^d)$. Moreover
\begin{equation}\label{eq-law} \widehat P \circ (Y_t , \widehat \Gamma (t,Y_t, d\al)) = P\circ (X_t, \E^P [\Gamma_t (d\al) | X_t]), \quad t \in [0,T].\end{equation} 
Let $P^* := \widehat P \circ (\widehat \Gamma(t,Y_t)dt,Y)^{-1}$. Hence $P^*\in \cR(\mu)$ and it satisfies conditions~(M.\ref{pfcond}) and (M.\ref{Markoviancond}) by construction. Moreover we have that
\begin{align*}
J(\mu,P^*)&=\E^{\widehat P}\left[\int_{[0,T]\times A}f(t,Y_t,\mu_t,\al)\widehat\Gamma(t,Y_t)(d\al)dt + g(Y_T,\mu_T) \right]\\
	&\overset{(a)}{=}\E^{ P}\left[\int_{[0,T]\times A}f(t,X_t,\mu_t,\al)\widehat\Gamma(t,X_t)(d\al)dt + g(X_T,\mu_T) \right]\\
	&\overset{(b)}{=}\E^{ P}\left[\int_{[0,T]\times A}\E^{ P}\left[f(t,X_t,\mu_t,\al)\Gamma_t(d\al) \bigg\lvert X_{t}\right] dt + g(X_T,\mu_T) \right]\\
	&\overset{(c)}{=}\E^{P}\left[\int_{[0,T]\times A}f(t,X_t,\mu_t,\al)\Gamma_t(d\al)dt + g(X_T,\mu_T ) \right]\\
	&=J(\mu ,P),
\end{align*}
where equality (a) follows from the equivalent distribution of the processes involved, i.e. $\widehat P \circ Y^{-1}_t = P \circ X^{-1}_t$ for all $t\in [0,T]$, equality (b) is provided by \eqref{eq-law} and equality (c) is just the tower property of conditional expectations. Therefore $P^*$ satisfies condition~(M.\ref{ott}) so that $P^* \in \cR^*(\mu)$. By proceeding as in the proof of \cite[Corollary 3.8]{La}, we can easily conclude that $P^* \in \cR^* (\mu^*)$ with $\mu^* := P^* \circ X^{-1}$, i.e. $P^*$ is a Markovian relaxed MFG solution. \medskip

For the existence of a \emph{strict} Markovian MFG solution, assume that the set $K(t,x,\mu)$ defined in \eqref{eq:setK} is convex for all $(t,x,\mu) \in [0,T] \times \mathbb R^d \times \cP^p (\R^d)$. Hence by applying the same arguments as in the second part of the proof of Theorem 3.7 in \cite{La}, we get the existence of a measurable function $\widehat \gamma: [0,T] \times \mathbb R^d \to A$ such that $\widehat \Gamma (t,x)(d\al) = \delta_{\widehat \gamma(t,x)}(d\al)$.
\end{proof}

\begin{remark}
It is worth noticing that property \eqref{hyp-kurtzstock} is guaranteed, for instance, whenever the coefficients $b,\sigma, \beta$ are bounded in all their variables. In order to prove the existence of a Markovian MFG solution beyond that assumption, one could extend the mimicking approach taken in \cite{BS} to a jump-diffusion model like ours. However, such an approach, which consists in solving the problem first in discrete time and then pass to the limit along finer and finer partition, is quite a delicate one and it seems to rely heavily on path continuity. Finally, analogue mimicking results have been proved in \cite{BC} for jump-diffusions dynamics. Unfortunately, their assumptions on the jump measure are too strong for our setting and not so easy to check when compared to Kurtz and Stockbridge \cite{kurtz1998existence} results.    
\end{remark}


\section{Application: a toy model for an illiquid inter-bank market}
\label{sec:ToyModel}

In this section we consider a symmetric $n$-player game with controlled jumps fitting the general framework studied in the previous sections. The model shows the economic relevance of the general MFG considered in the previous section and, furthermore, it is simple enough to allow for explicit computations of the equilibrium in $n$-player case as well as when $n \to \infty$. The convergence to the MFG solution is also studied. The model is in our opinion an interesting variation of the systemic risk model proposed by Carmona and coauthors in \cite{CFS}, parsimoniously modified for including some illiquidity phenomenon.

More precisely, we consider $n \ge 2$ banks which lend to and borrow from a central bank in an interbank borrowing market. The aim of each bank is to keep its monetary reserves away from critical levels. After the financial crisis of 2008, banks are required by international regulation to store an adequate amount of liquid assets, cash, to manage possible market stress. At the same time, holding too much cash is costly for banks because of its low return. We will use the average monetary reserves of the system as  the benchmark for the reserves' level of each bank and so the cost function will penalize every deviation from this mean value as in \cite{CFS}.

The main difference with the model in \cite{CFS} is that therein the banks can control their reserves continuously over time, i.e. they can choose a rate at which they lend or borrow money, while in the present paper the interbank market is \emph{illiquid}. This means that the banks can borrow or lend money only at some exogenously given instants, that are modelled as jump times of a Poisson process with a certain intensity $\lambda >0$. The intensity can be viewed as an health indicator of the whole system: for instance, when the intensity is low, the probability of being able to control the reserves in the next instant is also low, the system is becoming very illiquid. This kind of situation typically arises during financial crises.

Let us turn to the mathematical description of the model. Let $X^i=(X^i_t)_{t\in[0,T]}$  denote the monetary log-reserves of bank $i$, whose evolution is given by
\begin{equation}
\label{diffpro}
dX_t^i=\frac{a}{n} \sum_{j=1}^n (X_t^j-X_t^i)\,dt +\sigma\,dW_t^i + \gamma_{t} ^i dN^i_t \,,\quad i=1,\dots,n,
\end{equation} 
where $(W^1,\ldots,W^n)$ is an $n$-dimensional Brownian motion and $(N^1,\ldots, N^n)$ is an $n$-dimensional Poisson process, each component with a constant intensity $\lambda >0$. We assume that the system starts at time $t=0$ from i.i.d. random variables $X_0^i=\xi^i$ such that $\mathbb E[\xi^i]=0$ for all $i=1,\ldots, n$, and that initial conditions, Brownian motions and Poisson processes are all independent.

The main difference with respect to \cite{CFS} is that the control $\gamma_t ^i$ appears only in the jump component, hence the bank $i$ cannot change its reserves continuously over time but only at the jump times of the Poisson process $N^i$. Notice that while the banks can borrow/lend money at different times, being the $N^i$'s independent, the jump intensity $\lambda$ is the same for each bank. 

We denote by $\bar{X}_t$ the empirical mean of the monetary log-reserves, i.e.
\[
\bar X_t=\frac{1}{n}\sum_{i=1}^n X_t^i\,, 
\]
hence the dynamics of bank $i$ reserves can be rewritten in the mean field form as
\[
dX_t^i= a(\bar X_t -X_t^i) dt+\sigma\,dW_t^i + \gamma_{t} ^i dN^i_t \,,\quad i=1,\dots,n,
\]
and  the dynamics of the average state, $\bar X_t$, can be expressed as
\[
d\bar X_t =\frac{1}{n}\sum_{k=1}^n dX^k_t =\frac{\lambda}{n}\sum_{k=1}^n\gamma_t^k\,dt+\frac{\sigma}{n}\sum_{k=1}^n dW^k_t+\frac{1}{n}\sum_{k=1}^n\gamma_{t}^k\,d\widetilde N^k_t .
\]
The dynamics of the monetary reserves are coupled through their drifts by means of the average state of the system as in \cite{CFS}. Let $X=(X^1,\ldots,X^n)$ and $\gamma=(\gamma^1,\dots,\gamma^n)$. Bank $i$ controls the size of the jumps (if any) at time $t$, $\gamma_t^i$, in order to minimize a cost functional $J^i$ given by
\[
J^i(\gamma)=J^i(\gamma^1,\ldots,\gamma^n)=\mathbb{E}\bigg[\int_0^T f^i(X_t,\gamma_t^i)\,dN^i _t+g^i(X_T)\bigg] = \mathbb{E}\bigg[\int_0^T \lambda f^i(X_t,\gamma_t^i)\,dt+g^i(X_T)\bigg]\,,
\]
where $f^i$ and $g^i$ are quadratic functions as in \cite{CFS}, namely the running cost function $f^i \colon \R^n\times \R \to \R$ is given by
\[
f(x,\gamma^i) = \frac{1}{2}(\gamma^i)^2-\theta\gamma^i(\bar x -x^i)+\frac{\varepsilon}{2}(\bar x-x^i)^2 ,
\]
where $\bar x=\frac{1}{n}\sum_{i=1}^n x^i$ and the terminal cost function $g^i \colon \R^n \to \R $ is
\[
g^i ( x) = \frac{c}{2}(\bar x -x^i)^2.
\]
Note that both player $i$'s cost functions depend on the other players' strategies through the mean, i.e. $f^i(x,\gamma^i)=f(\bar x, x^i, \gamma^i)$ and $g^i(x)=g(\bar x, x^i)$.
The parameter $\theta>0$ is to control the incentive to borrowing or lending: the bank $i$ will want to borrow (i.e. $\gamma^i_t>0$) if $X^i_t$ is smaller than the empirical mean $\bar X_t$ and lend (i.e. $\gamma^i_t<0$) if $X^i_t$ is greater than the mean $\bar X_t$. We have taken the shape of the objective functions from \cite{CFS} (see therein for more details on the financial interpretation of such cost functions).

The parameters $\varepsilon$ and $c$ are strictly positive, so that the quadratic terms $(\bar x - x^i)^2$ in both costs penalize departure from the average. Moreover we assume that 
\[
\theta^2\le\varepsilon ,
\]
which guarantees the convexity of $f^i(x,\gamma)$ in both variables.

\begin{remark}\label{rmk:model-vs-theory}
Observe that the interbank illiquid model of this section fits the general theory developed in the previous sections only in the case $\varepsilon =0$ as running costs which are quadratic in the measure $\mu$ are ruled out by Assumption (\ref{ass:A}.\ref{ass:a:growth:f:g}). However, since one of the main goal of this section is giving an explicit characterization of a Nash equilibrium (in the $n$-player game) and an MFG solution (for the limit problem), we will be using other techniques based on forward/backward SDEs, allowing to treat the case $\varepsilon >0$ as well without any additional effort. Indeed, the approach followed in the previous sections is very powerful in giving abstract existence results, while not very suitable to compute the equilibria. 
\end{remark}

In the next sub-sections, we will compute a Nash equilibria in open-loop as well as closed-loop form (see Remark \ref{rem:CLp} below). Moreover, we will study the corresponding MFG when $n \to \infty$, whose solution will be strict Markovian. Our approach is heavily based on BSDE and it follows closely the one in \cite{CFS}. Nonetheless, we give all details at least in the open-loop case for reader's convenience. Finally, we will conclude with some simulations and some financial comments on the model. 

\subsection{Nash equilibrium in open-loop strategies}
\label{ssec:OLp}
We are searching for a Nash equilibrium among all admissible open-loop strategies $\gamma_t=\{\gamma^i_t, i=1,\dots,n\}$, that are real-valued predictable processes satisfying the integrability condition  $\mathbb{E}\big[\int_0^T | \gamma^i_t | dt\big]<\infty$ for all $i=1,\dots,n$. We denote the set of all such admissible controls by $\cA$.
  
For the problem of the $i$-th bank, we consider the Hamiltonian
\[
H^i (t,x,\gamma,y^i,q^i,r^i):[0,T]\times\mathbb{R}^n \times \R^n \times\R^n \times\R^{n\times n}\times \R^{n\times n}\to\mathbb{R}
\]
defined by
\begin{align}
\label{HamiA}
H^i(t,x,\gamma,y^i,q^i,r^i)&=\lambda f^i(x,\gamma)+(a(\bar x-x)+\lambda\gamma)\cdot y^i+\sigma \tr( q^i)+   \gamma\cdot \text{diag}(r^i) \notag\\
&=\lambda\left(\frac{(\gamma^i)^2}{2}-\theta(\bar x-x^i)\gamma^i+\frac{\varepsilon}{2}(\bar x-x^i)^2\right)+\sum_{k=1}^n \left[a(\bar x-x^k)+\lambda \gamma^k\right]y^{i,k}\\
&\hspace{5cm}+\sigma\sum_{k=1}^n q^{i,k,k}+ \sum_{k=1}^n \gamma^kr^{i,k,k}\,.\notag
\end{align}
The adjoint processes $Y^i_t=\{Y^{i,k}_t : k=1,\dots,n\}$, $Q^i_t=\{Q^{i,k,j}_t : k,j=1,\dots,n\}$ and $R^i_t=\{R^{i,k,j}_t : k,j=1,\ldots,n\}$ are defined as the solutions of the following BSDEs with jumps 
\begin{equation}
\label{adjiA}
\begin{cases}
dY_t ^{i,k}=-\frac{\partial H^i(t,X_t,\gamma_t,Y^i_t,Q^i_t,R^i_t)}{\partial x^k}\,dt+\sum_{j=1}^n
Q^{i,k,j}_t \,dW^j_t+\sum_{j=1}^n R^{i,k,j}_{t}\,d\widetilde N^j_t\\
Y^{i,k}_T=\frac{\partial g^i}{\partial x^k} (X_T)\,,
\end{cases}
\end{equation}
for $k=1,\dots,n$.

In order to find a candidate for the optimal control $\hat \gamma^i$, it suffices to minimize the Hamiltonian $H^i$ with respect to $\gamma^i$, leading to
\begin{equation}
\label{canA}
\hat \gamma^i=\theta(\bar x-x^i)- y^{i,i} - \frac{1}{\lambda} r^{i,i,i} \,.
\end{equation}
In order to prove that $\hat\gamma=(\hat\gamma^1,\dots,\hat\gamma^n)$ is a Nash equilibrium we show that when the other players $j \neq i$ are following $\hat\gamma^j$, then $\hat\gamma^i$ is the best response of player $i$. To do that, we need to solve the BSDE with jumps \eqref{adjiA}. It is natural to consider the ansatz
\begin{equation}
\label{ansA}
Y_t^{i,k}=\bigg(\frac{1}{n}-\delta_{i,k}\bigg)(\bar X_t-X^i_t)\phi_t ,
\end{equation}
where $\delta_{i,j}$ is the Kronecker delta and $\phi$ is a deterministic scalar function of class $C^1$, satisfying the terminal condition $\phi_T=c$, so that $Y^{i,k}_T=\frac{\partial g^i}{\partial x^k} (X_T)=c (\frac{1}{n}-\delta_{i,k} )(\bar X_T-X^i_T)$\,.
Differentiating the ansatz, we have that $Y^{i,k}$ solves the following SDE:
\begin{align}
dY_t^{i,k}&=\left(\frac{1}{n}-\delta_{i,k}\right)\phi_t d(\bar X_t-X^i_t)+\bigg(\frac{1}{n}-\delta_{i,k}\bigg)(\bar X_t-X^i_t)\dot\phi_t\,dt\notag\\
&=\bigg(\frac{1}{n}-\delta_{i,k}\bigg)\big[\lambda\phi_t(\bar \gamma_t-\gamma^i_t)+(\dot\phi_t-a\phi_t)(\bar X_t-X^i_t)\big]\,dt+\left(\frac{1}{n}-\delta_{i,k}\right)\phi_t\sigma\left(\frac{1}{n}\sum_{j=1}^n dW^j_t-dW^i_t\right)\notag\\
\label{deransA}&\hspace{0.3cm}+\bigg(\frac{1}{n}-\delta_{i,k}\bigg)\phi_{t} \left( \frac{1}{n}\sum_{j=1}^n \gamma^j_{t}\,d\widetilde N^j_t-\gamma^i_{t}\,d\widetilde N_t^i\right)\,,
\end{align}
where $\bar\gamma_t=\frac{1}{n}\sum_{k=1}^n \gamma^k_t$ denotes the average value of all control processes.
Comparing \eqref{adjiA} and \eqref{deransA} under the ansatz \eqref{ansA} yields
\begin{gather}
Q^{i,k,j}_t=\sigma\left(\frac{1}{n}-\delta_{i,k}\right)\left(\frac{1}{n}-\delta_{i,j}\right)\phi_t	\label{ansQ},\\
R^{i,k,j}_t= \left(\frac{1}{n}-\delta_{i,k}\right)\left(\frac{1}{n}-\delta_{i,j}\right)\phi_{t} \gamma^j_t ,\label{AnsR}
\end{gather}
for all $k,j=1,\dots,n$.

Moreover, by \eqref{ansA} and \eqref{AnsR}, it follows that the candidate $\hat \gamma^i$ given in $\eqref{canA}$ solves
\[
\hat \gamma^i_t=\theta(\bar X_{t-}-X_{t-}^i)-\bigg(\frac{1}{n}-1\bigg)\phi_{t}(\bar X_{t-}-X^i_{t-}) -\frac{1}{\lambda }\left(\frac{1}{n}-1\right)^2\phi_{t} \hat \gamma^i_t\,,
\]
and therefore the candidate optimal best response $\hat\gamma^i$ turns out to be
\begin{equation}
\label{can+ansA}
\hat \gamma^i_t=\frac{\theta+\left(1-\frac{1}{n}\right)\phi_{t}}{1+\frac{1}{\lambda}\left(1-\frac{1}{n}\right)^2\phi_{t}}(\bar X_{t-}-X^i_{t-})\,.
\end{equation}
To complete the description of $\hat \gamma^i$, we need to provide a characterization of the function $\phi$. From the definition of the Hamiltonian $H^i$ in \eqref{HamiA} it follows that the drift coefficient in the SDE \eqref{adjiA} for $Y$ as adjoint process is given by
\[
-\frac{\partial H^i(t,x,\gamma,y^i,q^i,r^i)}{\partial x^k}=\lambda\theta\bigg(\frac{1}{n}-\delta_{i,k}\bigg) \gamma^i-\lambda\varepsilon\bigg(\frac{1}{n}-\delta_{i,k}\bigg)(\bar x-x^i)-\frac{a}{n}\sum_{j=1}^n (y^{i,j}-y^{i,k})\,,
\]
and therefore, under the ansatz \eqref{ansA} and the related implications, equation \eqref{adjiA} becomes
\begin{eqnarray}
dY^{i,k}_t &=& \left(\frac{1}{n}-\delta_{i,k}\right) \left[\frac{\lambda\theta^2+\lambda\theta\left(1-\frac{1}{n}\right)\phi_t}{1+\frac{1}{\lambda}\left(1-\frac{1}{n}\right)^2\phi_t}-\lambda\varepsilon + a\phi_t\right](\bar X_t-X^i_t)\,dt\\
&& +\sum_{j=1}^n \left( Q^{i,k,j}_t \,dW^j_t+R^{i,j,k}_{t}\,d\widetilde N^j_t\right)\label{adjiA+ansA}\;. \nonumber
\end{eqnarray}
Since both equations \eqref{deransA} and \eqref{adjiA+ansA} hold simultaneously, we have that the following equality holds
\begin{gather*}
\dot\phi-a\phi-\frac{\theta+ \left(1-\frac{1}{n}\right)\phi}{1+\frac{1}{\lambda}\left(1-\frac{1}{n}\right)^2\phi}\phi=
\frac{\lambda\theta^2+\lambda\theta\left(1-\frac{1}{n}\right)\phi}{1+\frac{1}{\lambda}\left(1-\frac{1}{n}\right)^2\phi}-\lambda\varepsilon + a\phi_t
\end{gather*}
and this implies that $\phi_t$ must solve the ODE
\begin{multline}
\label{odephi}
 \left(1+\frac{1}{\lambda}\left(1-\frac{1}{n}\right)^2 \phi_t \right)\dot\phi_t = \left[\lambda+\frac{2a}{\lambda}\left(1-\frac{1}{n}\right)\right]\left(1-\frac{1}{n}\right)  \phi_t^2\\
  +\left[\lambda\theta\left(2-\frac{1}{n}\right)-\varepsilon \left(1-\frac{1}{n}\right)^2+2a \right]\phi_t+\lambda(\theta^2-\varepsilon)\,, 
\end{multline}
with terminal condition $\phi_T =c$. For more details on how such an ODE can be solved at least in implicit form we refer the reader to Appendix \ref{sec:Behaviour}.


\subsection{Approximate Nash equilibria}
\label{ssec:LimitCase}
We conclude the theoretical study of our model by computing the solution of the MFG obtained from the $n$-player game as $n\to \infty$. In particular, we will see that the $n$-player Nash equilibria computed in the previous sections tend towards the MFG solution.

Let $m \in \cD$ be a given \cadlag function, which represents a candidate for the limit of $\mathbb E[\bar X_t]$ when $n \to \infty$. Consider the following one-player minimization problem
\[
\inf_{\gamma}\mathbb{E}\left[\int_0^T \lambda\left(\frac{\gamma_t^2}{2}-\theta\gamma_t(m(t)-X_t)+\frac{\varepsilon}{2}(m(t)-X_t)^2\right)\,dt+\frac{c}{2}(m(T)-X_T)^2\right]
\]
subject to the dynamics
\begin{equation}
\label{SDEXlc}
dX_t=[a(m(t)-X_t)+\lambda \gamma_t]\,dt+\sigma \,dW_t+ \gamma_{t}\,d\widetilde N_t, \quad X_0=x_0,
\end{equation}
where $W$ is a standard Brownian motion and $N$ a Poisson process with intensity $\lambda>0$. The processes $W$ and  $N$ are assumed to be independent. Then, for a given \cadlag function $m \in \cD$, the optimal control $\hat\gamma$ is chosen in the class of admissible controls, which are real-valued predictable processes $\gamma$ such that $\mathbb E[\int_0 ^T |\gamma_t| dt] < \infty$. Moreover, in order for $\hat \gamma$ to be a MFG solution, the mean field condition has to be satisfied, i.e. $\E[X^{\hat\gamma}_t]=m(t)$ for a.e. $t\in [0,T]$.

As in the $n$-player case, we solve the problem via the Pontryagin maximum principle. In this case, the Hamiltonian is given by
\[
H(t,x,y,q,r,\gamma)=\lambda\left(\frac{\gamma^2}{2}-\theta\gamma(m(t)-x)+\frac{\varepsilon}{2}(m(t)-x)^2\right)+[a(m(t)-x)+\lambda\gamma]y+\sigma q+  \gamma r
\]
and the first order condition implies that the candidate optimal strategy is
\[
\hat\gamma=\theta(m(t)-x)- y - \frac{r}{\lambda}\,.
\]
The corresponding adjoint forward-backward equations are as follows: the forward process $X$ solves
\begin{equation}\label{FSDEX}
\begin{cases}
dX_t=[(a+\lambda\theta)(m(t)-X_t)-\lambda Y_t - R_t]\,dt+\sigma\,dW_t+ \left [\theta(m(t-)-X_{t-})-\left(Y_{t-}+\frac{R_{t}}{\lambda}\right)\right]\,d\widetilde N_t\\
X_0=x_0
\end{cases}
\end{equation}
while the BSDE for the adjoint processes $(Y,Q,R)$ is given by
\begin{equation}\label{BSDEY}
\begin{cases}
dY_t=\left[(a+\lambda\theta)Y_t+ \theta R_t + \lambda(\varepsilon- \theta^2)(m(t)-X_t)\right]\,dt+Q_t\,dW_t+R_{t}\,d\widetilde N_t
\\
Y_T=c(X_T-m(T)).
\end{cases}
\end{equation}
Note that this time the optimization problem is one-dimensional and therefore $Y$, $Q$ and $R$ are real-valued stochastic processes.

Since equations~\eqref{FSDEX}-\eqref{BSDEY} are linear we can firstly solve for the expectations $\E[X_t]$  and $\E[Y_t]$. By taking the expectation in both sides of equations~\eqref{FSDEX} and \eqref{BSDEY} and using the martingale property of the integrals with respect to the Brownian motion and the compensated Poisson process, we have
\begin{gather*}
d\E[X_t]=\left( (a+\lambda\theta)(m(t)-\E[X_t])-\lambda \E[Y_t] - \E[R_{t }]\right) \,dt ,
\end{gather*}
which becomes
\begin{gather*}
dm(t)=(-\lambda \E[Y_t] - \E[R_{t}])\,dt ,\end{gather*}
under the mean field condition $m(t)=\E[X_t]$. 

In order to solve the BSDE \eqref{BSDEY}, we make the same ansatz as before, that is $Y_t=-\phi_t(m(t)-X_t)$, which by identification implies
\[
Q_t=\sigma\phi_t\,,\quad R_t= \frac{\theta+\phi_t}{1+\frac{1}{\lambda}\phi_t}\phi_t(m(t-)-X_{t-}),
\]
for some deterministic function $\phi$ of class $C^1$ with final value $\phi_T=c$, so that $Y_T=c(X_T-m(T))$ as required by the BSDE \eqref{BSDEY}. Notice that these processes can be obtained by those in \eqref{ansA}-\eqref{ansQ}-\eqref{AnsR} by letting $n\to\infty$. 

If we plug the ansatz in the BSDE \eqref{BSDEY}, we find that the process $Y_t$ solves the following SDE
\begin{equation}
\label{SDE1}
dY_t=\left(\lambda(\varepsilon-  \theta^2) - (a +\lambda \theta) \phi_t + \theta\frac{\theta+\phi_t}{1+\frac{1}{\lambda}\phi_t}\phi_t \right)(m(t) - X_t) dt+Q_t\,dW_t+R_{t}\,d\widetilde N_t
\end{equation}
and, at the same time,  by differentiating the ansatz $Y_t = -\phi_t (m(t)-X_t)$, we have that
\begin{align*}
dY_t&
=\left(-\dot\phi_t(m(t)-X_t) + \phi_t \left( \lambda \E[Y_t] + \E[R_t] + (a+\lambda \theta)(m(t)-X_t) - \lambda Y_t - R_t \right)  \right) dt\\
&\quad +Q_t\,dW_t+ R_{t} d\widetilde N_t
\\
&=\left( -\dot \phi_t + \phi_t \left( a+\lambda \theta + \lambda \phi_t - \phi_t \frac{\theta + \phi_t}{1+\frac{1}{\lambda}\phi_t} \right) \right)(m(t)-X_t)dt+ Q_t dW_t+ R_{t} d\widetilde N_t ,
\end{align*}
where once again we used the identity $m(t)=\E[X_t]$ and the equalities $\E[Y_t]=0=\E[R_t]$, which are due to the fact that both processes are proportional to $m(t)-\E[X_t]$.
By matching the two SDEs for $Y$,  we find that $\phi_t$ solves the following Cauchy problem
\begin{equation}
\label{phiLC}
\begin{cases}
\left(1+\frac{1}{\lambda}\phi_t\right)\dot\phi_t= \left(\lambda + \frac{2a}{\lambda}\right)\phi^2_t +(2(a+\lambda \theta) -\varepsilon)\phi_t-\lambda(\varepsilon-\theta^2)
\\
\phi_T =c
\end{cases}
\end{equation}
and therefore the optimal control turns out to be
\begin{equation}\label{mfg-appl} \hat \gamma_t = \frac{\theta+\phi_{t}}{1+ \frac{1}{\lambda}\phi_{t}}(\mathbb E[X_{t-}]-X_{t-}), \quad t\in [0,T].
\end{equation}
Observe that this can also be obtained as limit for $n \to \infty$ of the Nash equilibrium computed before in the $n$-player game (see equations~\eqref{can+ansA} and \eqref{odephi}). Figure \ref{fig:PhiN} displays the behaviour of $\phi$, solution of the ODE \eqref{odephi}, for different values of players' number $n$. As $n$ increases, the graph of $\phi=\phi(n)$ quickly converges to the solution we found in the game with an infinite number of players, given in equation \eqref{phiLC}.

\begin{figure}[h!]
    \centering
    \includegraphics[width=1\textwidth]{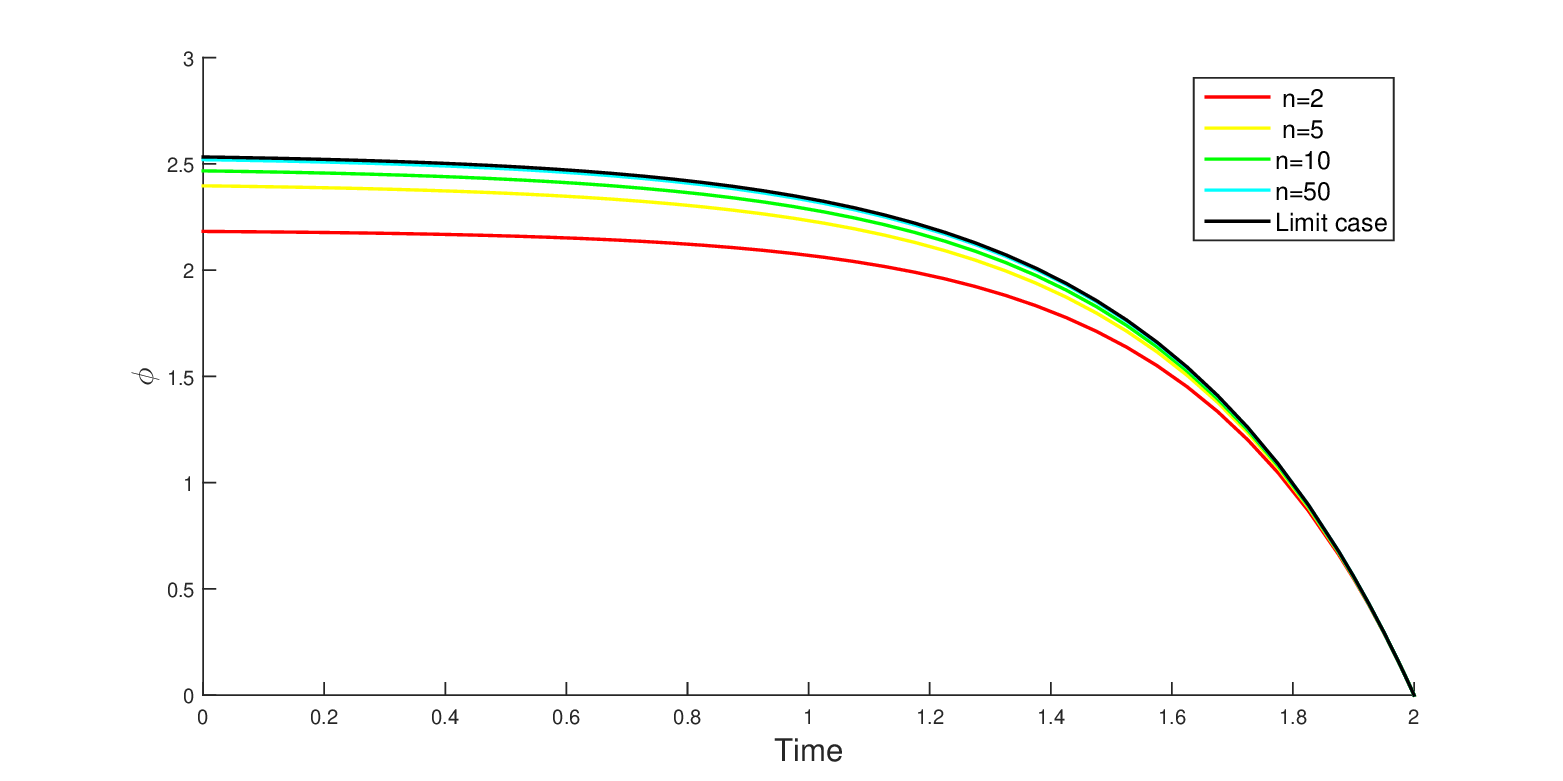}
    \caption{Plots of $\phi$, solution of the ODE \eqref{odephi} for different values of $n$. Model's parameter: $T=2$, $a=1$, $\theta=1$, $\varepsilon=10$, $\lambda=0.7$, $c=0$ }
    \label{fig:PhiN}
\end{figure}

\begin{remark}
\label{rem:CLp} In the closed-loop case, each player $i$ has complete information on the private states of the other participants, and therefore his best reply is chosen among all Markovian strategies of the form $\gamma^{i} (t,X_t)$, $t \in [0,T]$, for some real-valued function $\gamma_t ^i (t,x)$.
Following the same approach as before, based on the Pontryagin maximum principle and BSDEs with jumps, we obtain a Nash equilibrium in closed-loop form $\tilde \gamma^{i} $ given by
\begin{equation}
\label{eq:CLsol}
\tilde \gamma^i _t = \tilde \gamma^i (t, X_{t-}) =\frac{\theta+\left(1-\frac{1}{n}\right)\eta_{t}}{1+\frac{1}{\lambda}\left(1-\frac{1}{n}\right)^2\eta_{t}}(\bar X_{t-}-X^i_{t-}) , \quad i=1,\ldots,n,
\end{equation}
where the function $\eta_t$, analogously as in the open loop case, solves a suitable ODE with the same terminal condition as before, i.e. $\eta_T =c$. We do not develop further on this since, qualitatively speaking, the two equilibria, open loop and closed loop, produces the same kind of behaviour, which is illustrated by some numerical experiments in the next section in the open loop case. Moreover, in the closed loop case too we have convergence of the $n$-player Nash equilibrium above towards the MFG solution \eqref{mfg-appl} as $n \to \infty$. Therefore, we have decided to focus on the differences between our model and \cite{CFS} which are due to illiquidity. 
\end{remark}

\subsection{Simulations}
Previous computations in Section~\ref{ssec:OLp} and Remark \ref{rem:CLp} shows that the open-loop optimal strategies (see equation~\eqref{can+ansA}) have the form
\[
\hat\gamma^i _t=\frac{\theta+\left(1-\frac{1}{n}\right)\phi_t}{1+\frac{1}{\lambda}\left(1-\frac{1}{n}\right)^2\phi_t}(\bar X_{t-}-X^i_{t-}), \quad t\in[0,T]\, ,
\] 
for some deterministic function $\phi$ solving a suitable ODE. In this section, we examine the dependence on the parameters of the model (in particular, $\lambda$) of the open-loop strategies and log-reserves at the Nash equilibrium computed before. The same analysis applied to the closed-loop case would lead to the same qualitative conclusions (cf. Remark \ref{rem:CLp}).

First, figure~\ref{fig:Scenario} shows a typical scenario of our model. It can be observed that the optimal strategy is such that at each jump time for the reserves of a given bank, i.e. when the bank can borrow or lend money, the reserves move closer to the average level $\bar X$ of the reserves in the system. This is clearly expected due to the fact that such an average is the benchmark for each bank. There are two reasons why they do not match exactly and the second one is linked to presence of jumps in the model. First, reaching $\bar X$ can be too costly. Second, the choice of each bank, say bank $1$, at time $t-$ depends on the difference between its reserve $X^1_{t-}$ and the average reserves $\bar X_{t-}$ immediately before time $t$, with the aim of reducing such difference. But at the same time, $\bar X$ might have a jump at time $t$, as a consequence of the jump in the reserves of bank $1$. So even if $X^1_{t}=\bar X_{t-}$  we could have $X^1_{t} \neq \bar X_{t}$.
\begin{figure}[h!]
    \centering
    \includegraphics[width=0.9\textwidth]{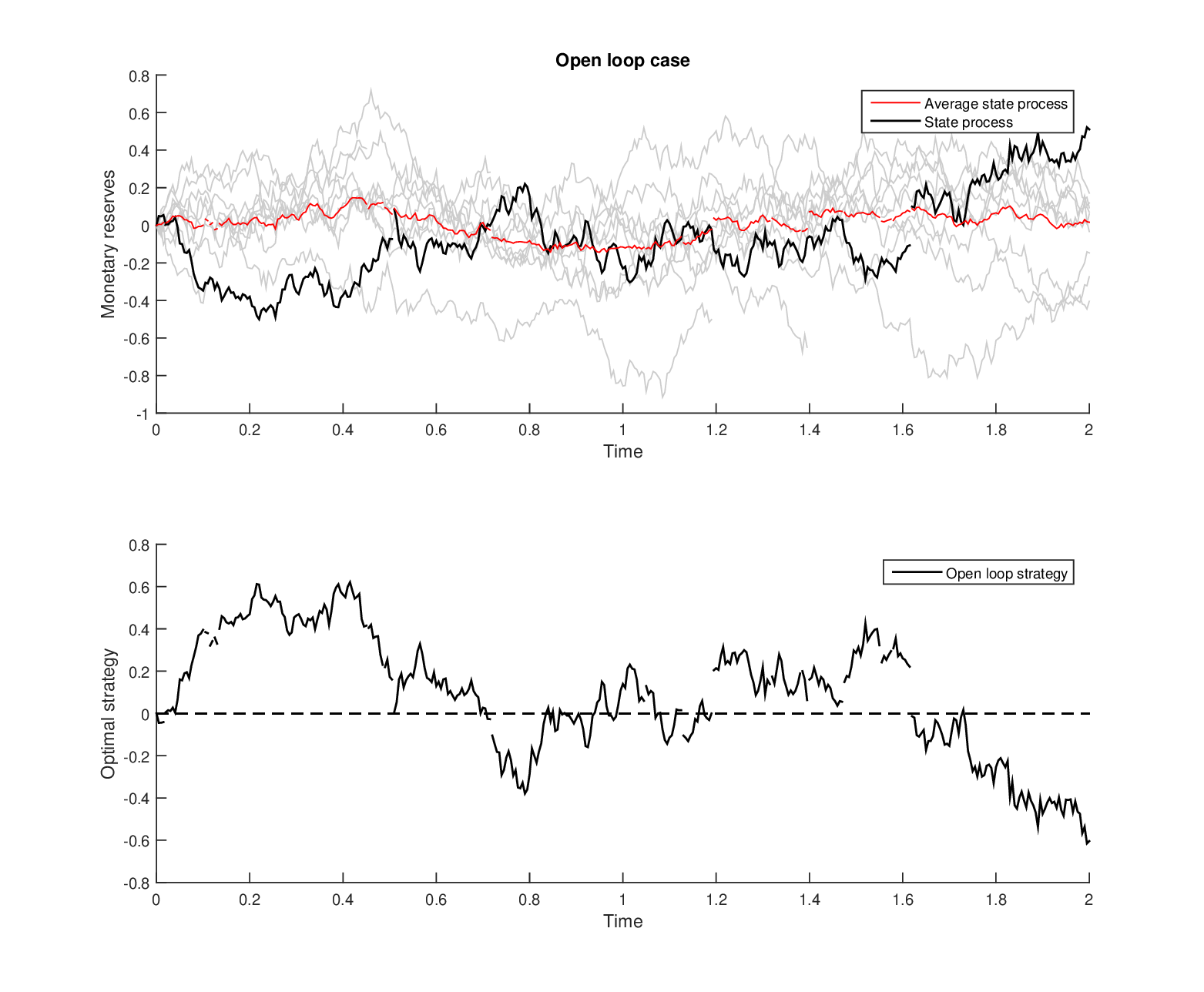}
    \caption{Typical scenario. Model's parameters: $n=10$, $T=2$, $a=1$, $\sigma=0.8$, $X_0=0$, $\theta=1$, $\varepsilon=10$, $c=0$, $\lambda=1.2$.}
\label{fig:Scenario}
\end{figure}

Now, let us focus on the variation of the equilibrium strategies due to changes in the intensity $\lambda$ of the Poisson processes, which we recall it represents the liquidity parameter of the inter-bank market. It is more convenient for the analysis to consider the function $\psi_t\colon[0,T]\to\mathbb{R}$ defined as
\[
\psi_t=\frac{\theta+\left(1-\frac{1}{n}\right)\phi_t}{1+\frac{1}{\lambda}\left(1-\frac{1}{n}\right)^2\phi_t} , \quad t\in [0,T],
\]
so that the open-loop optimal strategy can be expressed as $\hat\gamma^i _t=\psi_t (\bar X_{t-}-X^i_{t-})$, for $t\in[0,T]$. Hence, whenever one of the Poisson processes, say $N^i$, jumps, bank $i$ would modify its reserves by an amount $\hat \gamma^i _t$ which is proportional to the difference $\bar X_{t-} - X_{t-}^i$ just before the jump, with a proportionality factor $\psi_t$.

Then, routine computation shows that $\psi$ solves the following ODE:
\begin{align*}
\left(1-\frac{1}{\lambda}\left(1-\frac{1}{n}\right)\theta\right)^2 \dot\psi_t =\left[ \lambda+\frac{2a}{\lambda}\left(1-\frac{1}{n}\right) \right] \left(1-\frac{1}{\lambda}\left(1-\frac{1}{n}\right)\psi_t\right) (\psi_t-\theta)^2&\\
+\left[\lambda\theta\left(2-\frac{1}{n}\right) -\varepsilon\left( 1-\frac{1}{n} \right)^2 +2a\right] \left(1-\frac{1}{\lambda}\left(1-\frac{1}{n}\right)\psi_t\right)^2 (\psi_t-\theta)&\\
+\left[\left(1-\frac{1}{n}\right) \lambda(\theta^2-\varepsilon)\right] \left(1-\frac{1}{\lambda}\left(1-\frac{1}{n}\right)\psi_t\right)^3 ,
\end{align*}
with final value $\psi_T = \frac{\theta+\left(1-\frac{1}{n}\right)c}{1+\frac{1}{\lambda}\left(1-\frac{1}{n}\right)^2c}$. Notice that the terminal value is increasing in $\lambda$.

All our numerical experiments revealed that such a monotonicity behaviour propagates to the whole time interval, i.e. the proportionality factor $\psi_t$ is increasing in $\lambda$ for all $t \in [0,T]$. Here, in figures \ref{fig:psiL} to \ref{fig:psiLN100c1}, we show only the behaviour of $\psi$ as function of time $t \in [0,T]$ with $T=2$, $n \in \{10,100\}$, $c \in \{0,1\}$, and more importantly for different values of $\lambda$. Moreover, observe that (see Fig~\ref{fig:psiLc1}) the final value $\psi_T$ depends on the parameter $\lambda$ whenever $c$ is different than zero.

In the Nash equilibrium we have found, when $\lambda$ is small, hence the interbank market is very illiquid in the sense that banks will have (in expectation) very few possibilities to change their reserves, the reserves will change very little proportionally to $(\bar X_{t-} - X_{t-} ^i)$. On the other hand, when $\lambda$ is large, so that in expectation banks will have many occasions to lend/borrow money from the central bank, changes in their reserves will be very big proportionally to $(\bar X_{t-} - X_{t-} ^i)$. Therefore, focusing on the first case, we notice that instead of compensating the lack of liquidity ($\lambda$ small), banks seem to amplify it by borrowing and lending very little.

Another interesting feature one can notice from the figures is that when $\lambda$ is small, the proportionality factor $\psi_t$ is increasing in time. When the market is very illiquid, there are very few possibility for the banks to change their reserves during the time period, so that when the maturity $T$ is approaching, the banks knowing that they are running out of time to move their reserves closer to the average reserve $\bar X$, they amplify their efforts, whence an increasing $\psi_t$. An analogue interpretation can be provided for the opposite situation of a time-decreasing $\psi_t$ when $\lambda$ is large.

\begin{figure}[h!]
\centering
\includegraphics[width=\textwidth]{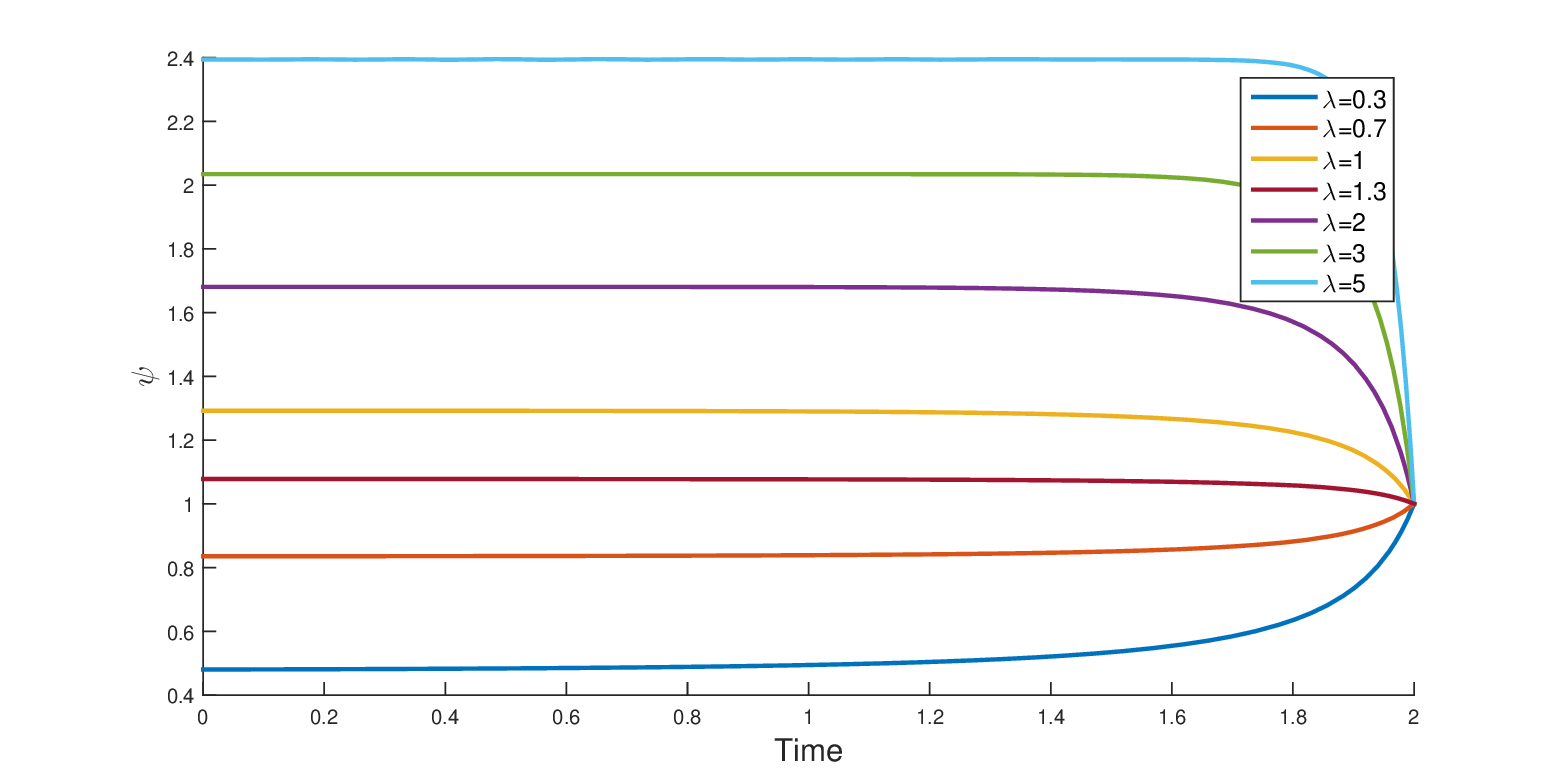}
\caption{Different values of $\lambda$. Model's parameters: $n=10$, $T=2$, $a=1$, $\theta=1$, $\varepsilon=10$, $c=0$.}
\label{fig:psiL}
\end{figure}

\begin{figure}[h!]
\centering
\includegraphics[width=\textwidth]{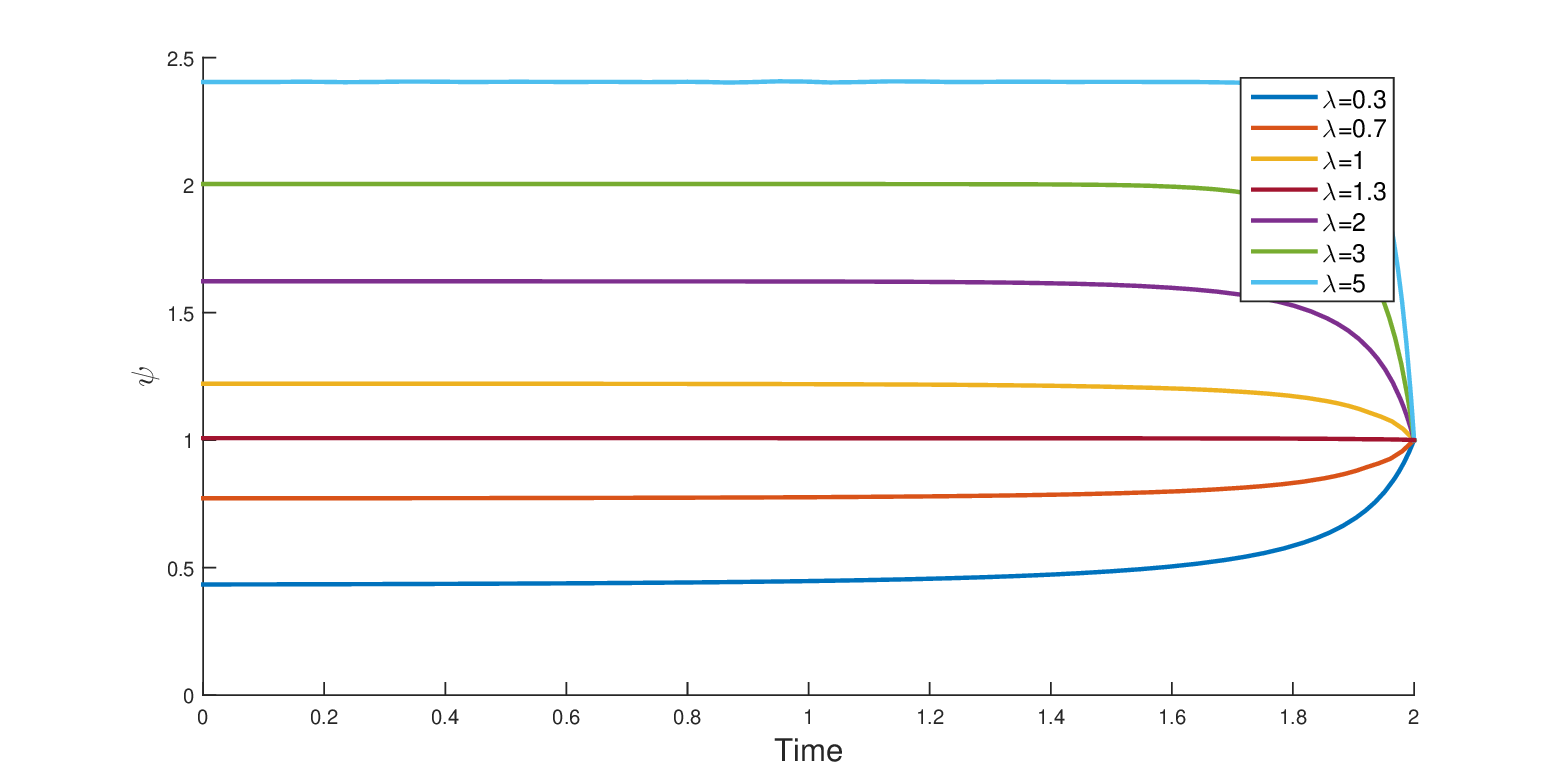}
\caption{Different values of $\lambda$. Model's parameters: $n=100$, $T=2$, $a=1$, $\theta=1$, $\varepsilon=10$, $c=0$.}
\label{fig:psiLN100c0}
\end{figure}

\begin{figure}[h!]
\centering
\includegraphics[width=\textwidth]{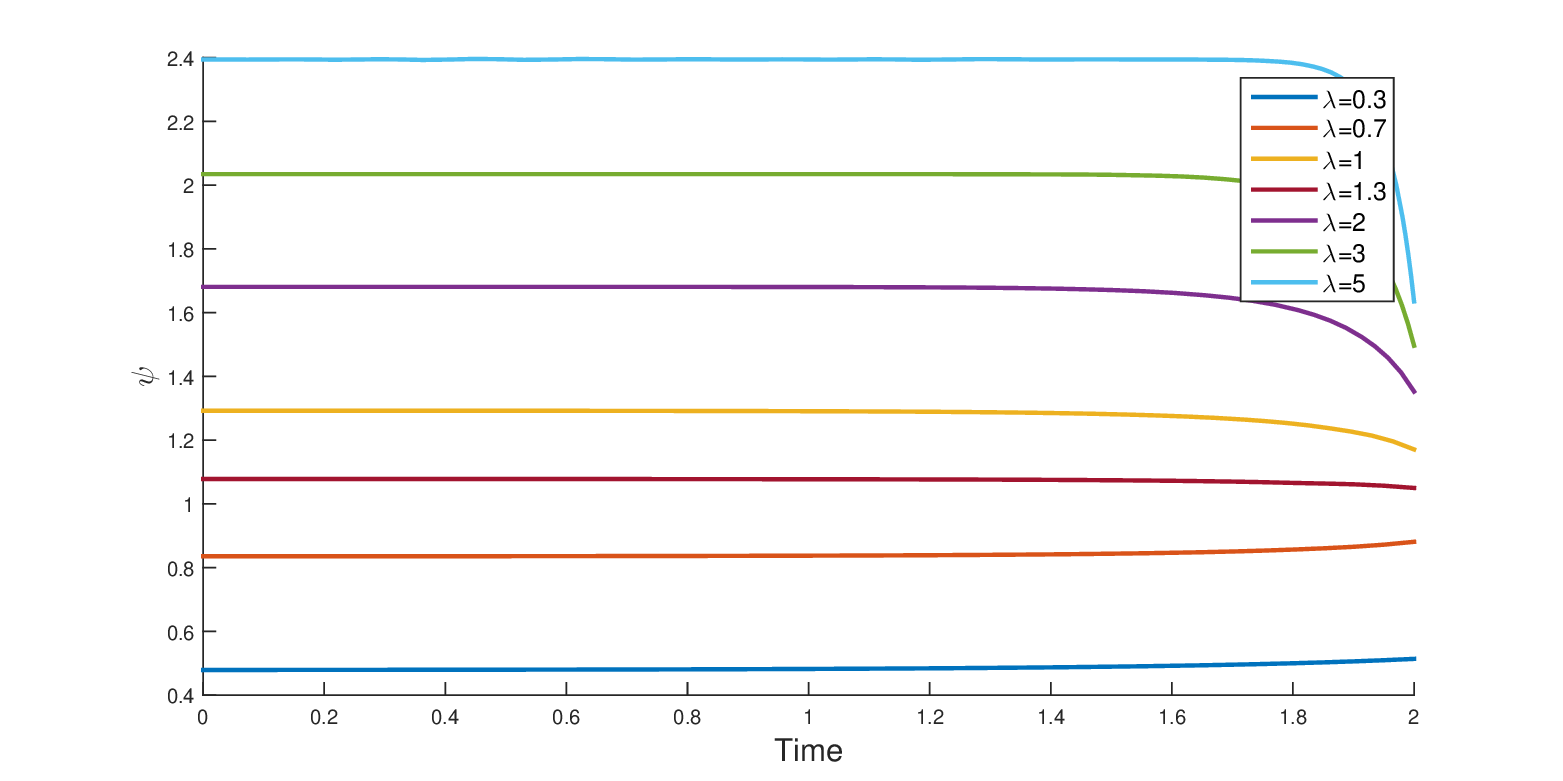}
\caption{Different values of $\lambda$. Model's parameters: $n=10$, $T=2$, $a=1$, $\theta=1$, $\varepsilon=10$, $c=1$.}
\label{fig:psiLc1}
\end{figure}

\begin{figure}[h!]
\centering
\includegraphics[width=\textwidth]{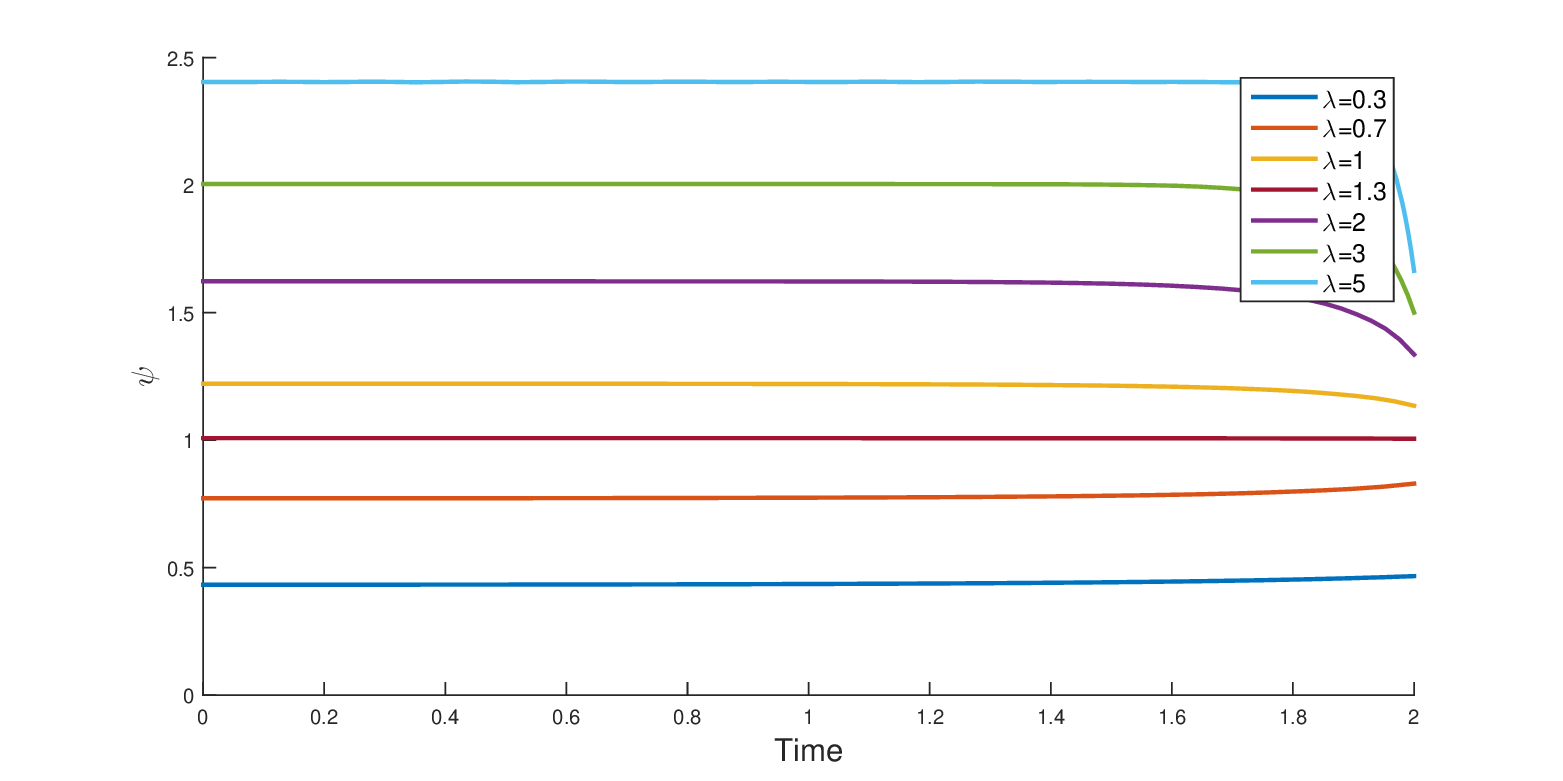}
\caption{Different values of $\lambda$. Model's parameters: $n=100$, $T=2$, $a=1$, $\theta=1$, $\varepsilon=10$, $c=1$.}
\label{fig:psiLN100c1}
\end{figure}

\section{Conclusions}
In this paper, we have generalised the controlled martingale problem approach developed in \cite{La} to MFGs with a multidimensional state variable following a jump-diffusion dynamics, where drift, diffusion coefficient and jump sizes are controlled. Some extra Lipschitz continuity in the measures for costs and coefficients was needed, due to presence of jumps in the flows of measures. Slightly more precisely, we have established existence of relaxed MFG solutions in the unbounded case as well of strict Markovian solutions under a sort boundedness assumption on the data of the problem. Moreover, using forward/backward SDE methods, we have studied a toy model for an illiquid interbank model in the spirit of \cite{CFS}. We have provided an explicit Nash equilibrium for the $n$-player game and the corresponding solution for the limit MFG and performed some numerics illustrating the properties of the equilibrium. Many further extensions are possible, here we mention the one that appears as the most appealing in view of applications: considering in the state variable dynamics a jump process with a compensator depending on the state, measure as well as strategies would pave the road to models whose behaviour of the players can affect directly or indirectly the frequency of jumps arrival. In this case, one would have to deal with non-standard SDE with jumps (as, e.g., in \cite{JP}), so that a different approach would need to be used. Hence, we have decided to postpone this further generalisation to future research.


\section*{Acknowledgements}
L. Campi and L. Di Persio would like to thank the FBK (Fondazione Bruno Kessler) - CIRM (Centro Internazionale per la Ricerca Matematica) for funding their ``Research in Pairs'' project ``McKean-Vlasov dynamics with L\'evy noise with applications to systemic risk'' for the period November 1-8, 2015, with the active participation of C. Benazzoli as well. Moreover L. Di Persio and C. Benazzoli would like to thank the ``Gruppo Nazionale per l'Analisi Matematica, la Probabilità e le loro Applicazioni (GNAMPA)'' for funding the projects ``Stochastic Partial Differential Equations and Stochastic Optimal Transport with Applications to Mathematical Finance'', coordinated by L. Di Persio (2016), and ``Metodi di controllo ottimo stocastico per l’analisi di problem debt management'', coordinated by A. Marigonda (2017). Both such projects supported the research object of the present paper.


\noindent
\printbibliography

\begin{appendix}

\section{Intermediate results for the existence of MFG solutions}
\label{sec:intermediateresults}

This appendix collects useful results which were used in the proof of Theorems~\ref{thm:relaxedMFGs} and \ref{thm:MarkovianMFGsol}.

\subsection{A-priori estimates for the state variables} In the following, we will write $\abs{Y}^*_t$ as a shortcut for $\sup_{s\in[0, t]} \abs{Y_s}$. The next lemma is fairly standard (compare, e.g., Lacker's Lemma 4.3), nonetheless with give all details for reader's convenience.
\begin{lemma}
\label{stime}
Let $\bar p \in [p,p']$. Under Assumption \ref{ass:A}, there exists a constant $C=C(T, c_1,\chi, \bar p)$ such that for any $\mu\in \cP^p (\cD)$ and $P\in\cR(\mu)$ we have
\begin{equation}\label{bound-X}
\E^P\left[\left(\abs{X}^*_T\right)^{\bar p}\right]\le C \left(1+ \| \mu\|_T ^{\bar p} + \E^P \int_0 ^T | \Gamma_t |^{\bar p} dt \right).
\end{equation}
As a consequence, $P\in \cP^p (\Omega[A])$. Furthermore, if $\mu=P\circ X^{-1}$, we have
\[ \| \mu \|_T ^{\bar p} = \E^P\left[\left(\abs{X}^*_T\right)^{\bar p}\right] \le C \left( 1+ \E^P\left[ \int_0 ^T | \Gamma_t |^{\bar p} dt \right] \right).\]
\end{lemma}
\begin{proof} Consider a given probability measure $\mu\in\cP^p (\cD)$ and a related admissible law $P\in \cR(\mu)$. 
By Lemma~\ref{defR2}, there exists a constant $C$ such that 
\begin{align}\label{estimate1}
\abs{X_t}^{\bar p} & \le C\abs{X_0}^{\bar p}+C\abs{\int_0^t \int_A b(s,X_s,\mu_s, \al) \Gamma_s (d\al)ds}^{\bar p}+C\abs{ \int_0^t \int_A \sigma(s,X_s, \mu_s, \al) M(ds,d\al)}^{\bar p} \\
& \quad + C\abs{ \int_0^t\int_A \beta(s,X_{s-},\mu_{s-},\al )\widetilde{ \mathcal N}(ds,d\al)}^{\bar p} \,. \nonumber
\end{align}
In what follows, the value of the constant $C$ may change from line to line, however we will indicate what it depends on.
By Jensen's inequality and using the growth conditions on $b$ in Assumption \ref{ass:A}, we have for all $t \in[0,T]$
\begin{align*}
\abs{\int_0^t \int_A b(s,X_s, \mu_s, \al)\Gamma_s (d\al) ds}^{\bar p} &\le C_t \int_0^t \int_A \sup_{0\le u \le s}\abs{b(u,X_u, \mu_u, \al)}^{\bar p} \Gamma_s (d\al) ds\\
	&\le C_{t, \bar p} \int_0^t \int_A c_1^{\bar p} \left(1+ \left( |X|^*_s\right)^{\bar p}+ \| \mu\|_s ^{\bar p} + |\al|^{\bar p} \right)  \Gamma_s (d\al) ds\,.
\end{align*}

By the Burkholder-Davis-Gundy inequality (see \cite[Theorem 48, Ch. IV.4]{Protter}), the expected supremum of the It\^o integral in \eqref{estimate1} can be bounded as follows 
\begin{align*}
\E^P\left[\left(\abs{\int_0^\cdot \!\! \int_A \sigma(s,X_s, \mu_s, \al)M(ds,d\al)}^*_t\right)^{\bar p} \right] \le  C_{\bar p} \E\left[\left(\int_0^t \!\! \int_A \sup_{0\le u \le s} \abs{\sigma (u,X_u,\mu_u,\al)}^2 \Gamma_s (d\al) ds\right)^{{\bar p}/2} \right]\\
 \le C_{t,\bar p}  \E \left[ \int_0^t \int_A c_1^{\bar p/2}\left( 1+ |X|_s ^* + \left( \int_{\cD} (|z|_s^*) ^p \mu(dz)\right)^{1/p}  +|\al| \right)^{\bar p /2} \Gamma_s (d\al) ds \right].
\end{align*}
Lastly, we have to consider the integral $I_t=\int_{[0,t]\times A}\beta(s,X_{s-},\mu_{s-},\al)\widetilde {\mathcal N}(ds,d\al)$ in \eqref{estimate1}. Using the Burkholder-Davis-Gundy inequality again, see \cite[Theorem 48, Ch. IV.4]{Protter}, it follows that 
\begin{align*}
\E^P & \left[\left( \abs{\int_0^\cdot \int_A \beta(s,X_{s-},\mu_{s-},\al)\widetilde{\mathcal N}(ds,d\al)}^*_t\right)^{\bar p} \right]\\
& \le C_{\bar p} \E \left[\left(\int_0^t\int_A\sup_{0\le u \le s} \abs{\beta(u,\mu_{u-},X_{u-},\al)}^{2}\lambda(s)\Gamma_s (d\al)\,ds\right)^{\bar p/2}\right],
\end{align*}
and since $\beta(t,x,\mu,\al)$ has linear growth in $(x,\mu,\al)$ uniformly in $t$, it is found that $$\E^P\left[\left(\abs{I}^*_t\right)^{\bar p} \right]\le C \E \left[\int_0^t\int_A  c_1^{\bar p/2 + 1}\left( 1+ |X|_s ^* + \left( \int_{\cD} (|z|^*_s) ^p \mu(dz)\right)^{1/p} +|\al | \right)^{\bar p/2}\Gamma_s (d\al)\,ds\right].$$
Observe that the previous two bounds hold if $\bar{p}\ge2$. On the other hand, if $\bar p\in[1,2)$ the conclusion still holds since $\abs{y}^{\bar p/2}\le 1 + |y|^{\bar p}$ and then arguing as before.

Combining all the previous estimates, we get that there exists a positive constant $ C=C(t,c_1,\chi, \bar p)$ such that 
\[
\E^P\left[\left(\abs{X}^*_t\right)^{\bar p} \right]\le C \E^P\left[ 1+ |X_0 |^{\bar p} + \int_0 ^t \left(1+ \| \mu\|_s ^{\bar p} + |\Gamma_s |^{\bar p} \right)ds \right], \quad t \in [0,T]. 
\]
Hence the estimates (\ref{bound-X}) follows from an application of Gronwall's lemma. As a consequence, when $\mu = P\circ X^{-1}$, we have
\[ \| \mu \|_t ^{\bar p} = \E^P [(| X|_t ^*)^{\bar p}] \le C \E^P \left[ |X_0| ^{\bar p} + \int_0 ^t (1+2 \| \mu \| _s ^{\bar p} + | \Gamma_s| ^{\bar p}) ds\right],\]
so that another application of Gronwall's lemma gives the second estimate.  
\end{proof}

\subsection{Some continuity results}
The following result is an extension of Corollary A.5 in \cite{La}, where the space of continuous functions is replaced with the Skorokhod space $\cD ([0,T] ; E)$ of all right-continuous with left limit functions taking values in some metric space $(E,\rho)$.  

\begin{lemma}\label{Lemma:Cont}
Let $(E,\rho)$ be a complete separable metric space. Let $\phi: [0,T]\times E \times A \to \mathbb R$ be a jointly measurable function in all its variables and jointly continuous in $(x,\al ) \in E \times A$ for each $t \in [0,T]$. Assume that for some constant $c>0$ and some $x_0 \in E$ one of the following two properties is satisfied\begin{enumerate}
\item \label{usc} $\phi(t,x,\al) \le c(1+\rho^p (x,x_0) + |\al |^p)$, for all $(t,x,\al) \in [0,T] \times E \times A$;
\item \label{cont} $| \phi(t,x,\al) | \le c(1+\rho^p (x,x_0) + |\al |^p)$, for all $(t,x,\al) \in [0,T] \times E \times A$.
\end{enumerate}
Hence, if \eqref{usc} (resp. \eqref{cont}) is fulfilled, the following function
\begin{equation}\label{phi-dq} \cD ([0,T] ; E) \times \V[A] \ni (x,q) \mapsto \int_{[0,T] \times A} \phi(t,x(t),\al) q(dt,d\al)\end{equation}
is upper semicontinuous (resp. continuous). 
\end{lemma}
\begin{proof}
We prove first that the function
\begin{equation}\label{eta} \cD ([0,T] ; E) \times \V[A] \ni (x,q) \mapsto \eta (dt,d\al,de) := \frac{1}{T} q(dt,d\al) \delta_{x(t)} (de) \in \cP^p ([0,T] \times E \times A)\end{equation}
is jointly continuous. In order to do that we adapt and expand the arguments provided in the proof of \cite[Lemma 1.5]{kurtz1992averaging}. Using \cite[Prop. A.1]{La} it suffices to show that when $(x^n ,q^n) \to (x,q)$ in $\cD ([0,T] ; E) \times \V[A]$ as $n\to \infty$, we have $\int \phi d\eta^n \to \int \phi d\eta$ for all continuous functions $\phi: [0,T] \times E \times A \to \mathbb R$ such that $| \phi(t,x,\al) | \le c(1+\rho^p (x,x_0) + |\al|^p)$, for all $(t,x,\al) \in [0,T] \times E \times A$. We use the notation $\eta^n$ for the measure associated to $(x^n, q^n)$ as in \eqref{eta}. 

By definition of Skorokhod topology, there exists a sequence of time changes $\tau_n(t)$, i.e. strictly increasing continuous functions mapping $[0,T]$ onto $[0,T]$, 
such that
\begin{equation} \sup_{t \in [0,T]} |\tau_n (t) - t| \to 0 \quad \textrm{and} \quad \lim_{n \to \infty} \sup_{t \in [0,T]} \rho (x^n (\tau_n (t)), x(t)) = 0. \label{unif-conv}\end{equation}
Let $\bar q^n ([0,t] \times H) := q^n ([0,\tau_n (t)]\times H)$ for all $t\in [0,T]$ and any Borel set $H \subset A$. Since $q^n$ is tight, for all $\epsilon >0$ there exists a compact set $K \subset A$ such that $\sup_n q^n ([0,T]\times K) < \epsilon$. Clearly, the same property holds for $\bar q^n$ as well as $q$. Moreover, being $\phi$ continuous, we have
\begin{equation} \sup_{(t,\alpha) \in [0,T] \times K} |\phi(\tau_n (t), x^n (\tau_n(t)), \alpha) - \phi(t,x(t),\alpha)| \to 0, \quad n\to \infty. \label{conv-phi} \end{equation}
Therefore we can show
\begin{align*} 
\int_{[0,T]\times A} \phi(t,x^n (t), \al ) q^n (dt,d\al )  & =  \int_{[0,T]\times A} \phi(t,x^n (\tau_n(t)), \al ) \bar q^n (dt,d\al )\\
&  \to \int \phi(t,x (t), \al ) q(dt,d\al ), \quad n \to \infty .\end{align*}
Indeed, letting $K^c$ denote the complement of $K$, we have
\begin{align*}
& \abs{\int_{[0,T]\times A} \phi(t,x^n (\tau_n(t)), \al ) \bar q^n (dt,d\al ) - \int_{[0,T]\times A} \phi(t,x (t), \al ) q(dt,d\al ) } \\
&\le \left | \int_{[0,T]\times K} \phi(t,x^n (\tau_n(t)), \al ) \bar q^n (dt,d\al ) - \int_{[0,T]\times K} \phi(t,x (t), \al ) q(dt,d\alpha) \right | \\
&\quad  + \left | \int_{[0,T]\times K^c} \phi(t,x^n (\tau_n(t)), \al ) \bar q^n (dt,d\al ) - \int_{[0,T]\times K^c} \phi(t,x (t), \al ) q(dt,d\alpha) \right | 
\end{align*}
where the second summand above can be made arbitrarily small, uniformly in $n$, since both integrands are in their respective $L^p$ spaces due to the polynomial growth of the function $\phi$. Furthermore, the first summand can be handled as follows:
\begin{align*}
& \left | \int_{[0,T]\times K} \phi(t,x^n (\tau_n(t)), \al ) \bar q^n (dt,d\al ) - \int_{[0,T]\times K} \phi(t,x (t), \al ) q(dt,d\alpha) \right | \\
& \le \int_{[0,T]\times K} \left | \phi(t,x^n (\tau_n(t)), \al ) - \phi(t,x(t),\al)\right | \bar q^n (dt,d\al )\\
& \quad + \left | \int_{[0,T]\times K} \phi(t,x (t), \al ) \bar q^n (dt,d\alpha) - \int_{[0,T]\times K} \phi(t,x (t), \al ) q (dt,d\alpha)  \right |
\end{align*}
where the first summand goes to zero thanks to \eqref{conv-phi}, while the second one does too due to $\bar q^n \to q$ in the (modified) $p$-Wasserstein distance defined in \eqref{2-modWass}.

Finally, the upper semicontinuity (resp. continuity) of the function \eqref{phi-dq} is obtained by applying \cite[Corollary A.4]{La} (resp. \cite[Lemma A.3]{La}). 
\end{proof}

\begin{lemma}\label{Lemma:cont-mu}
Let $\{\mu^n\}\subseteq\cP^p(\cD)$ a convergent sequence to $\mu\in\cP^p(\cD)$. 
Then, for any $q\ge1$
\begin{align*}
\int_0^T d_{W,p}(\mu^n_t,\mu_t)^q\,dt\to0, \quad n \to\infty.
\end{align*}
\end{lemma}
\begin{proof}
First, recall that convergence with respect to $d_{W,p}$ implies also weak convergence, hence
Skorokhod's representation theorem ensures that there exist $\cD$-valued random variables $X^n$ and $X$, defined on a common probability space $(\Omega,\mathcal{F}, P)$ such that
\[
\mu^n= P\circ (X^n)^{-1}\text{ and }\mu=P\circ X^{-1}\,,
\]
with
\[
d_{J_1}(X^n,X)\to 0, \quad P\text{ a.s.},
\]
where $d_{J_1}$ denotes any metric consistent with Skorokhod $J_1$-topology (see, e.g., \cite[Chp. 3, Sect. 5]{EK}).
By the triangular inequality, we obtain
\[
\abs{X^n_t-X_t}^p \le 2^p (d_{J_1}(X^n,0)^p + d_{J_1}(X,0)^p)
\]
and 
\[
 d_{J_1}(X^n,0)^p + d_{J_1}(X,0)^p \to  2 d_{J_1}(X,0)^p\,, \quad n \to \infty,
\]
where $d_{J_1}(X,0)^p=\left(\abs{X}^*_T\right)^p \in L^1(P)$, which does not depend on $t$.
Then, since the convergence $X^n\to X$ in $J_1$ implies convergence for a.e $t\in[0,T]$, by applying a slightly more general version of dominated convergence (e.g. \cite[Theorem 1.21]{kallenberg}), we have 
\[
\E\left[\abs{X^n_t-X_t}^p\right]\to 0\quad\text{a.e. } t\in[0,T]\,.
\]
Then, by applying dominated convergence once again in view of
\[
\E^P \left[	d_{J_1}(X^n,0)^p + d_{J_1}(X,0)^p \right] \to  2 \int_\cD d_{J_1}(x,0)^p \mu(dx)<\infty\,,
\]
we can conclude that
\[
\int_0^T\left(\E\left[\abs{X^n_t-X_t}^p\right]\right)^q\,dt\to 0\,, \quad n \to \infty. \qedhere
\]
\end{proof}

The two previous lemmas together with the assumptions of Lipschitz continuity in $\mu$ of $f$ and the growth conditions of both $f$ and $g$ (cf. Assumption \ref{ass:A}), yields the following continuity properties of the objective functional $J$ in \eqref{opJ}.

\begin{lemma}\label{Jcont}
The operator $J$ defined in \eqref{opJ} is upper semi-continuous under Assumption \ref{ass:A}. Moreover, if Assumption \ref{ass:B} is also in force, the operator $J$ is continuous.
\end{lemma}
\begin{proof} The upper semi-continuity is an easy consequence of Lemmas \ref{Lemma:Cont} and \ref{Lemma:cont-mu}, while the continuity follows from the compactness of $A$. More precisely, Lemma \ref{Lemma:Cont} is used to prove the semi-continuity of $\mathcal C^\mu (X,\Gamma)$ in $(X,\Gamma)$, while the one in $\mu$ is granted by Lemma \ref{Lemma:cont-mu}. \end{proof}


We conclude this part of the appendix with establishing a continuity property for the set-valued map $\cR$, giving the set of all admissible laws corresponding to measures $\mu \in \cP^p (\cD)$, as given in Definition \ref{defR}.  

\begin{lemma}\label{Rcont}
Under Assumptions \ref{ass:A} and \ref{ass:B}, the set-valued correspondence $\cR$ given in Definition \ref{defR} is continuous with relatively compact image $\cR(\cP^p (\cD))=\bigcup_{\mu\in\cP^p (\cD)}\cR(\mu)$ in $\cP^p(\Omega[A])$.
\end{lemma}
\begin{proof} We split the proof in three steps. \medskip

\emph{Step 1: the image of $\cR$ is relatively compact.}
Using \cite[Lemma A.2]{La}, we prove that $\cR(\cP^p (\cD))$ is relatively compact in $\cP^p (\Omega[A])$ by proving that $\{P\circ\Gamma^{-1} : P\in\cR(\cP^p (\Omega[A]))\}$ and $\{P\circ X^{-1} : P\in\cR(\cP^p (\Omega[A]))\}$ are relatively compact sets in $\cP^p (\V )$ and $\cP^p (\cD)$ respectively. The compactness of $\{P\circ\Gamma^{-1} : P\in\cR(\cP(\Omega[A]))\}$ in $\cP(\V )$ equipped with the $p$-Wasserstein metric follows from the compactness of $A$, and therefore of $\V $.
On the other hand, the compactness of $\{P\circ X^{-1} : P\in\cR(\cP^p (\Omega[A]))\}$ follows from Proposition~\ref{Qccompact}.\medskip

\emph{Step 2: $\cR$ is upper hemicontinuous.} In order to show that $\cR$ is upper hemicontinuous, we use the closed graph theorem (see, e.g., \cite[Theorem 17.11]{AlBo}) by proving that $\cR$ is closed, i.e. for each $\mu^n \to \mu\in \cP^p (\cD)$ and for each convergent sequence $P^n \to P$ with $P^n\in\cR(\mu^n)$, it holds that $P\in\cR(\mu)$.
According to Definition~\ref{defR}, we have to show that $P\circ X_0^{-1}=\chi$ and that $\cM^{\mu,\phi}$, defined as in \eqref{opM} on the canonical probability space $(\Omega[A],\mathcal{F},(\mathcal{F}_t)_{t\in[0,T]},P)$, is a $P$-martingale for all $\phi \in C_0 ^\infty$. 
\\
The first condition is satisfied since convergence in probability implies convergence in distribution and therefore $X_0\overset{d}{=}\lim_{n\to\infty} X^n_0$, whose law is given by $\chi$.
\\
Regarding the second condition, let $s,t \in [0,T]$ be such that $0\le s \le t\le T$, and let $h$ be a continuous, $\mathcal{F}_s$-measurable, bounded function. By the martingale property of $\cM^{\mu^n,\phi}$, we have $\E^P[(\cM^{\mu^n,\phi}_t-\cM^{\mu^n,\phi}_s)h]=0$ for all $n$. Hence to prove that $\cM^{\mu,\phi}$ is a $P$-martingale for all $\phi \in C_0 ^\infty$ it suffices to prove
\begin{equation}
\label{eq:limitmart}
\lim_{n\to\infty}\E^P[(\cM^{\mu^n,\phi}_t-\cM^{\mu^n,\phi}_s)h] = E^P[(\cM^{\mu,\phi}_t-\cM^{\mu,\phi}_s)h].
\end{equation}
First of all, note that the functional $\cM^{\mu^n,\phi}$ can be bounded, uniformly on $n$. Indeed, by Taylor's theorem, we have
\begin{align*}
\abs{L\phi(t,x,\mu,\Gamma)}&\le\norm{  b(t,x,\mu,\alpha)^\top D\phi(x)}_\infty +\frac{1}{2}\norm{ \tr\left[ \sigma\sigma^\top(t,x,\mu,\alpha)D^2\phi(x)\right]}_\infty \\
&\quad
 + \norm{  \left[\phi(x+\beta(t,x,\mu , \alpha))-\phi(x)- \beta(t,x,\mu ,\al)^\top D\phi (x)\right] \lambda(t) }_\infty\\
 &\le\norm{  b(t,x,\mu,\alpha)^\top D\phi(x)}_\infty +\frac{1}{2}\norm{ \tr\left[ \sigma\sigma^\top(t,x,\mu,\alpha)D^2\phi(x)\right]}_\infty \\
&\quad
 +\frac{1}{2} \norm{  \left[ \beta(t,x,\mu ,\al)^\top D^2\phi (x + \xi \beta(t,x,\mu ,\al)) \beta(t,x,\mu ,\al)\right] \lambda(t) }_\infty\\
&\le d ( c_1 + c^2_1) C_{\phi} ,
\end{align*}
where $\xi\in[0,1]$, $\norm{b}_\infty+\norm{\sigma}_\infty+\norm{\beta}_\infty \le c_1$, $\norm{D \phi}_\infty + \norm{D^2 \phi}_\infty\le C_{\phi}$. Therefore
\[
\| \cM^{\mu^n,\phi} \| \le \norm{\phi}_\infty +T d ( c_1 + c^2_1) C_{\phi} =: \bar C_\phi\,.
\] 
The global continuity of the functions $b$, $\sigma$, $\beta$ and $\lambda$ guarantees that also the function $L\phi$ is for all test functions $\phi$. Moreover, since, according to~Assumptions \ref{ass:A} and~\ref{ass:B}, all such coefficients are bounded, and  $b$, $\sigma$ and $\beta$ are Lipschitz continuous with respect to the variable $\mu$ uniformly in $(t,x,\alpha)$, then also $L\phi$ is Lipschitz with respect to $\mu\in\cP(\R)$. Indeed 
\begin{align*}
\label{opL}
|L\phi & (t,x,\mu, \al)-L\phi(t,x,\nu,\al)|
\\
&\le \big| \left(b(t,x,\mu, \alpha)-b(t,x,\nu,\alpha)\right)^\top D\phi(x)+\frac{1}{2}\tr\left[\sigma \sigma^\top(t,x,\mu,\alpha)-\sigma\sigma^\top (t,x,\nu,\alpha)\right]D^2 \phi(x)\\
&\quad + \big[\phi(x+\beta(t,x,\mu , \alpha))-\phi(x)- \beta(t,x,\mu ,\al)^\top D\phi (x) \\
& \quad -\phi(x+\beta(t,x,\nu , \alpha))+\phi(x) + \beta(t,x,\nu ,\al)^\top D\phi (x)\big] \lambda(t)\big|
\\
&\le \abs{b(t,x,\mu, \alpha)-b(t,x,\nu,\alpha)}^\top \abs{D\phi(x)}\\
&\quad+\frac{1}{2}\abs{\tr\left[(\sigma (t,x,\mu,\alpha)+\sigma (t,x,\nu,\alpha))(\sigma (t,x,\mu,\alpha) - \sigma(t,x,\nu,\alpha))^\top\right]}\abs{D^2 \phi(x)}\\
&\quad+ c_1\abs{\phi(x+\beta(t,x,\mu , \alpha))-\phi(x+\beta(t,x,\nu , \alpha))} +c_1 \abs{\beta(t,x,\mu ,\al)- \beta(t,x,\nu ,\al)}^\top \abs{ D\phi (x)} 
\\
&\le \abs{b(t,x,\mu, \alpha)-b(t,x,\nu,\alpha)}^\top \abs{D\phi(x)} \\
&\quad + \frac{1}{2}\norm{\sigma (t,x,\mu,\alpha)+\sigma (t,x,\nu,\alpha)}_\infty\abs{\sigma (t,x,\mu,\alpha) - \sigma(t,x,\nu,\alpha)}\abs{D^2\phi(x)}\\
&\quad + c_1 \abs{\beta(t,x,\mu, \alpha)-\beta(t,x,\nu,\alpha)}^\top  \left( \abs{D\phi(x+\xi\beta(t,x,\mu,\alpha))}+ \abs{D\phi(x)} \right)
\\
& \le (C_{\phi} + 2 c_1 C_{\phi} + 2 c_1 C_{\phi}) d_{W,p}(\mu,\nu) ,
\end{align*}
where $\xi\in[0,1]$. Therefore we can conclude that 
\[
(\mu,\Gamma,X)\mapsto\int L\phi(t,X_t,\mu_t,\alpha)\Gamma(d\alpha)dt
\]
is continuous. Indeed, the continuity with respect to $(X,\Gamma)$ is provided by Lemma~\ref{Lemma:Cont}, whereas the continuity with respect to $\mu$ is an application of Lemma~\ref{Lemma:cont-mu} since by the previous computation it follows that
\[
\int_A \abs{L\phi(t,X_t,\mu_t,\alpha)- L\phi(t,X_t,\nu_t,\alpha)}\Gamma(d\alpha)\le C d_{W,p}(\mu_t,\nu_t)\,.
\]
Therefore for each continuous bounded function $h$
\[
\lim_{n\to\infty}\E^{P^n}\left[\left(\int L\phi(s,X^n_s,\mu^n_s,\alpha)\Gamma^n_s(d\alpha)ds\right)h\right]=\E^P\left[\left(\int L\phi(s,X_s,\mu_s,\alpha)\Gamma_s(d\alpha)ds\right)h\right]
\]
and moreover, since $P^n\to P$ and $\phi$ is bounded and continuous, we have that
\[
\E^{P^n}[\phi(X^n_t)]=\int \phi(X^n_t)\,dP^n \to \int \phi(X_t)\,dP = \E^{P}[\phi(X_t)]\,.
\]
This means
\[
\E^P[(\cM^{\mu,\phi}_t-\cM^{\mu,\phi}_s)h]=\lim_{n\to\infty}\E^P[(\cM^{\mu_n,\phi}_t-\cM^{\mu_n,\phi}_s)h]=0,
\]
which implies $\E^P[\cM^{\mu,\phi}_t-\cM^{\mu,\phi}_s]=0$, i.e. $\cM^{\mu,\phi}$ is a martingale and $P\in\cR(\mu)$.\medskip

\emph{Step 3: $\cR$ is lower hemicontinuous.} We are left with proving that $\cR$ is also lower hemicontinuous. Let $\mu\in \cP^p (\cD)$ and $\mu^n$ be a sequence in the same space converging to $\mu$. Then, for every $P\in\cR(\mu)$ we need to exhibit a sequence  $P_n\in\cR(\mu^n)$ such that $P_n\to P$ in $\cP^p (\Omega[A])$. Let $(\Omega',\mathcal{F}',\F',P')$ be a filtered probability space, $M=(M^1,\dots,M^m)$ orthogonal continuous $\mathbb F^\prime$-martingale measures on $[0,T]\times A$ with intensity $\Gamma_t(da)dt$ and $\mathcal N$ a Poisson random measure on $[0,T] \times A$ with intensity measure $\Gamma_t(da)\lambda(t)dt$, where $\lambda(t)$ is a bounded function of time.
Let $X$ be the unique strong solution of the following SDE:
\begin{equation}
\label{Sde:X}
dX_t =\int_A b(t,X_t, \mu_t,\al)\Gamma_t(d\alpha) dt + \int_A \sigma(t,X_t,\mu_t,\alpha) M(dt,d\alpha) + \int_A \beta(t,X_{t-},\mu_{t-},\alpha)  \widetilde{\mathcal N}(dt,d\al)
\end{equation}
with initial condition $X_0$. Lipschitz continuity and growth conditions on the coefficients grant existence and uniqueness of strong solution of equation~\eqref{Sde:X}.
 Then by uniqueness we have $P'\circ(\Gamma,X)^{-1}=P$.
For each $n$, define $X^n$ as the process solving
\[
dX^n_t =\int_A b(t,X^n_t, \mu^n_t,\al)\Gamma_t(d\alpha) dt + \int_A \sigma(t,X^n_t,\mu^n_t,\alpha) M(dt,d\alpha) + \int_A \beta(t,X^n_{t-},\mu^n_{t-},\alpha)  \widetilde{\mathcal N}(dt,d\al).
\]
We want to show that
\begin{align}\label{Ediffp}
\E^{P'}\left[\left(\abs{X^n-X}_t^*\right)^p\right]\to0\,.
\end{align}
Let $\bar p=\max\{2,p\}$. Then, for a suitable positive constant $C>0$ we have
\begin{align}
\label{diseqXn}
\abs{X^n_t-X_t}^{\bar p}& \le C\abs{\int_0^t \int_A \abs{b(s,X^n_s,\mu^n_s,\alpha)-b(s,X_s,\mu_s,\alpha)}\,\Gamma_s(d\alpha)ds}^{\bar p} \\
& \quad +C\abs{\int_0^t\int_A\abs{\sigma(s,X^n_s,\mu^n_s,\alpha)-\sigma(s,X_s,\mu_s,\alpha)}\,M(ds,d\alpha)}^{\bar p} \nonumber \\
& \quad +C\abs{\int_0^t \int_A \abs{\beta(t,X^n_{t-},\mu^n_{t-},\alpha)-\beta(t,X_{t-},\mu_{t-},\alpha)}\widetilde{\mathcal N}(dt,d\al)}^{\bar p}. \nonumber
\end{align}
Since $b$ is Lipschitz continuous in $x$ and $\mu$, we have that
\begin{multline}
\label{diseq1}
\E^{P'}\left[\abs{\int_0^t \int_A \abs{b(s,X^n_s,\mu^n_s,\alpha)-b(s,X_s,\mu_s,\alpha)}\,\Gamma_s(d\alpha)ds}^{\bar p}\right]\\
\le C(c_1,\bar p,T)\left(\int_0^t \E^{P'} \left(\abs{X^n-X}^*_s\right)^{\bar p} ds + \int_0^t d_{W,p} ( \mu_s ^n , \mu_s)^{\bar p} ds\right)\,.
\end{multline} 
Regarding the stochastic integral in \eqref{diseqXn}, we can apply Jensen and  Burkholder-Davis-Gundy inequalities to obtain
\begin{multline}
 \E^{P'} \left[\left(\abs{\int_0^\cdot \int_A\abs{\sigma(s,X^n_s,\mu^n_s,\alpha)-\sigma(s,X_s,\mu_s,\alpha)}\,M(ds,d\alpha)}_t^*\right)^{\bar p}\right]
\\\le C(\bar p, c_1) \left(\int_0^t \E^{P'} \left(\abs{X^n-X}^*_s\right)^{\bar p} ds + \int_0^t d_{W,p} ( \mu_s ^n , \mu_s)^{\bar p} ds\right)\,.
\end{multline}
Lastly, applying the Burkholder-Davis-Gundy inequality to the Poisson measure integral in~\eqref{diseqXn} combined with the boundedness of the intensity function $\lambda(t)$, we have that
\begin{align*}
&\E^{P'}\left[	\left(\abs{\int_0^\cdot \int_A \abs{\beta(t,X^n_{t-},\mu^n_{t-},\alpha)-\beta(t,X_{t-},\mu_{t-},\alpha)}\widetilde{\mathcal N}(dt,d\al)}^*_t\right)^{\bar p}	\right]
\\
&\hspace{2cm}\le C(\bar p) \E\left[	\int_0^t \int_A \abs{\beta(s,X^n_{s-},\mu^n_{s-},\alpha)-\beta(s,X_{s-},\mu_{s-},\alpha)}^{\bar p} \Gamma_s(d\alpha)	\right]
\\
&\hspace{2cm}\le C(\bar p, c_1) \left(\int_0^t \E^{P'} \left[\left(\abs{X^n-X}^*_s\right)^{\bar p} \right] ds + \int_0^t d_{W,p} ( \mu_s ^n , \mu_s)^{\bar p} ds\right)\,.
\end{align*}
Notice that the last integral $\int_0^t d_{W,p} ( \mu_s ^n , \mu_s)^{\bar p} ds$ in all estimates above converges to zero as $n \to \infty$ due to Lemma \ref{Lemma:cont-mu}.

Therefore, combining the previous results and applying the Gronwall's inequality, we conclude that $\E^{P'}\left[\left(\abs{X^n-X}^*_T\right)^p\right]\to0$.

 By defining $P^n=P\circ(\Gamma,X^n)^{-1}$ we have found the desired sequence such that $P^n\to P$ in $\cP^p (\Omega[A])$. To conclude, we have to show that $P^n$ is indeed in $\cR(\mu^n)$. $P^n$ satisfies condition~\eqref{condin} in Definition~\ref{defR} by construction, and condition~\eqref{martingaleproperty} can be checked by applying It\^o's formula to $\phi(X_t ^n)$, for each $\phi\in C^\infty_b(\R^d )$.
\end{proof}


\subsection{A compactness result}
\begin{proposition}
\label{Qccompact}
Let $c>0$ be a given positive constant and let $p' >p \ge 1$. Let $\chi$ be a given initial law. Define $\mathcal Q_c \subset\cP^p (\Omega[A])$ as the set of laws $Q = P \circ (X,\Gamma)^{-1}$ of $\Omega[A]$-valued random variables defined on some filtered probability space $(\Omega,\mathcal{F},\F,P)$ with $\F=(\mathcal F_t)$, such that:
\begin{enumerate}
\item $dX_t= \int_A b(t,X_t, \mu_t , \al)\Gamma_t (d\al) dt + \int_A \sigma(t,X_t,\mu_t , \al) M(dt,d\al) +\int_{ A} \beta(t,X_{t-},\mu_{t-},\al )\widetilde{\mathcal N}(dt,d\al)$, where $M=(M^1,\ldots, M^m)$ are orthogonal $(\mathcal F_t)$-martingale measures on $A \times [0,T]$ with intensity $\Gamma_t (da)dt$ and $\mathcal N$ is a random measure with intensity $\Gamma_t (da)\lambda(t)dt$, where the function $\lambda : [0,T] \to \mathbb R$ is measurable and bounded;
\item $P\circ X_0^{-1} = \chi$;
\item \label{assbb} $(b,\sigma, \beta)\colon[0,T]\times \R^d \times \cP^p (\R^d) \times A \to \R^d \times \R^{d \times m}\times \R^d$ are measurable functions such that
\begin{equation} \label{growth-condition}
|b(t,x, \al )| + |\sigma \sigma^\top(t,x,\al )| + |\beta (t,x,\al)| \le c(1+|x|+|\al |),
\end{equation}
for all $(t,x,\al ) \in [0,T] \times \R^d \times A$;
\item $\E \left[ |X_0|^{p'} + \int_0 ^T |\Gamma_t|^{p'} dt \right] \le c$.
\end{enumerate} 
Hence $\mathcal Q_c$ is relatively compact in $\cP^p (\Omega[A])$.
\end{proposition}

\begin{proof}
Since $\cD$ is a Polish space under $J_1$ metric,
Prokhorov's theorem (cf. \cite[Theorem 5.1, Theorem 5.2]{Bill}) ensures that a family of probability measures on $\cD$ is relatively compact if and only if it is tight. In order to prove the tightness, we will use the Aldous's criterion provided in \cite[Theorem 16.10]{Bill}.
By proceeding as in the previous Lemma~\ref{stime}, there exists a constant $C=C(T,c,\chi)$ (uniform in $Q$) such that
\[
\E^Q\left[\left(\abs{X}^*_T\right)^{p}\right]\le C \E^Q \left[ 1+ |X_0 | ^{p'} + \int_0 ^T |\Gamma_t |^{p'} dt \right] ,
\]
which implies that $\E^Q [ (\abs{X}^*_T )^p ]$ is bounded by a constant which depends upon $Q$ only through the initial distribution $\chi$, which is the same for all laws in $\cQ$. Hence, we have
\begin{equation}
\label{EPnormX}
\sup_{Q\in\mathcal{Q}_c} \E^Q\left[\left(\abs{X}^*_T\right)^{p}\right]
<\infty\,.
\end{equation}
Then we are left with proving that
\begin{equation}
\label{Aldous}
\lim_{\delta\downarrow0}\sup_{Q\in\mathcal{Q}_c}\sup_{\tau \in \mathcal T_T} \E^Q \left[\abs{X_{(\tau+\delta)\wedge T}-X_\tau}^p\right]=0,
\end{equation}
where $\mathcal T_T$ denotes the family of all stopping times with values in $[0,T]$ a.s.
For each $Q\in\mathcal{Q}_c$ and each stopping time $\tau\in \mathcal T_T$, there exists a constant $\tilde C$ such that
\begin{align} \label{stimetau}
\begin{split}
\E^Q\left[\abs{X_{(\tau+\delta)\wedge T}-X_\tau}^p\right]&\le \tilde C\E^Q\left[\abs{\int_A \int_\tau^{(\tau+\delta)\wedge T}b(t,X_t, \al)\Gamma_t (d\al) dt}^{p}\right] \\
&\hspace{0.5cm} +\tilde C \E^Q \left[\abs{\int_A \int_\tau^{(\tau+\delta) \wedge T}\sigma(t,X_t ,\al)M(d\al, dt)}^{p}\right]\\
	&\hspace{0.5cm}+\tilde C\E^Q \left[\abs{\int_{[\tau,(\tau+\delta)\wedge T]\times A}\beta(t,X_{t-},\al) \widetilde{\mathcal N}(dt,d\al)}^{p}\right].
\end{split}
\end{align}
Without loss of generality, we may assume that $\delta\le1$. By applying Burkholder-Davis-Gundy inequality, the property \eqref{growth-condition} and the boundedness of the intensity function $\lambda(t)$ as in the proof of Lemma~\ref{stime}, there exists a constant $C$ such that, for all $Q \in \mathcal Q_c$ and all stopping times $\tau \in \mathcal T_T$, we have
\begin{align}\label{def:barC}
\begin{split}
\E^Q\left[\abs{X_{(\tau+\delta)\wedge T}-X_\tau}^p\right]&\le C\E^Q\left[\abs{\int_A \int_\tau^{(\tau+\delta)\wedge T}(1+|X|^*_T + |\al |)\Gamma_t (d\al) dt}^{p}\right] \\
&\hspace{0.5cm} + C \E^Q \left[\left(\int_A \int_\tau^{(\tau+\delta) \wedge T}(1+|X|^*_T + |\al |) \Gamma_t (d\al) dt \right)^{p/2}\right].
\end{split}
\end{align}
From this point onwards we can proceed as in the proof of \cite[Proposition B.4]{La}, which gives
\[
\lim_{\delta\downarrow 0}\sup_{Q\in\mathcal{Q}}\sup_{\tau \in \mathcal T_T} \E^Q\left[\abs{X_{(\tau+\delta)\wedge T}-X_\tau}^p\right]=0\,.
\]
Hence Aldous' criterion applies and the proof is completed.\end{proof}

\section{Details on the computation of $\phi$ in equation \eqref{odephi}}
\label{sec:Behaviour}\noindent
We want to solve the final value Cauchy problem
\[
\dot \phi_t=F(\phi_t ), \quad \phi_T=c>0 ,
\]
where
\[
F(u)=\frac{Au^2+Bu+C}{1+k u} .
\]
Then, if we perform the time reversal $\tau=T-t$ we can consider the equivalent Cauchy problem
\[
\dot \phi_\tau =-F(\phi_\tau ), \quad \phi_0 =c>0 .
\]
Moreover we look for a solution whose graph belongs to the domain $D_k := [0,T]\times(-\frac{1}{k},\infty)$, where $k:=\frac{1}{\lambda}\left(1-\frac{1}{n}\right)^2$, so that the optimal strategy given in \eqref{can+ansA} is well defined for all $t\in[0,T]$. Notice that $k\neq0$ if $\lambda > 0$ and $n \ge 2$, which we can safely assume to rule out trivialities.

Note that since $-F$ and $-\dot F$ are continuous functions (in the smaller domain $D_k$) then we can apply standard results for existence and uniqueness.
Therefore the function $\phi$ solves the following Cauchy problem
\[
(1+k\phi_t)\dot \phi_t=A\phi^2_t+B\phi_t+C, \quad \phi_T=c
\]
for appropriate values of $k,A,B,C$ which do not depend on time $t$:
\begin{align*}
A & := \left(\lambda +\frac{2a}{\lambda}\left(1-\frac{1}{n}\right)\right)\left(1-\frac{1}{n}\right)>0,\\
B & := \lambda\theta\left(2-\frac{1}{n}\right)-\varepsilon \left(1-\frac{1}{n}\right)^2+2a,\\
C & := \lambda(\theta^2-\varepsilon)<0.
\end{align*}
Let $$\omega(t):=\left(\phi_t +\frac{1}{k}\right)e^{-\frac{A}{k}t} .$$
Then $\omega$ solves the following ODE:
\[
\dot\omega \omega=F_1(t)\omega+F_0(t) ,
\] 
where
\begin{gather*}
F_0(t)=\left(\frac{C}{k}-\frac{B}{k^2}+\frac{A}{k^3}\right)e^{-2\frac{A}{k}t} ,\quad F_1(t)=\left(\frac{B}{k}-2\frac{A}{k^2}\right)e^{-\frac{A}{k}t}.
\end{gather*}
Lastly, by choosing $$\xi=\int F_1(t)\,dt=\left(\frac{2}{k}-\frac{B}{A}\right)e^{-\frac{A}{k}t} ,$$ we have that
\begin{equation}\label{eq-omega}
\omega(\xi)\dot\omega(\xi)=\omega(\xi)+K\xi ,
\end{equation}
where
\[
K=-\frac{(Ck^2-Bk+A)A}{(2A-kB)^2} .
\]
Equation \eqref{eq-omega} can be solved in parametric form as
\[
\xi=h \exp\left(-\int\dfrac{\tau}{\tau^2-\tau-K}\,d\tau\right) , \quad \omega= h \tau \exp\left(-\int\dfrac{\tau}{\tau^2-\tau-K}\,d\tau\right) ,
\]
where $h$ is a suitable constant. Therefore,
\[
\xi=h \exp\left(-\int\dfrac{\tau}{\tau^2-\tau-K}\,d\tau\right) , \quad 
\phi= \left(h \tau \exp\left(-\int\dfrac{\tau}{\tau^2-\tau-K}\,d\tau\right)\right)e^\frac{A}{k}t-\frac{1}{k}\, .
\]
Then, computing the integral in the definition of $\xi$ gives 
\[
\xi = h \exp \left(-\dfrac{\tan^{-1}\left(\frac{-1+2 \tau}{\sqrt{-1-4 K}}\right)}{\sqrt{-1-4 K}}\frac{1}{\sqrt{\tau^2-\tau-K}}\right).
\]
Finally, in order to compute $\phi$ explicitly, we need to invert this relation  $\xi \mapsto \xi(\tau)$ and plug it into $\omega$.

\end{appendix}

\end{document}